\documentclass[pagesize,pdftex]{scrartcl}
\usepackage[pagebackref=true]{hyperref}
\usepackage{amsmath}
\usepackage{amsthm}
\usepackage{amsfonts}
\usepackage{amssymb}
\usepackage{amsxtra}
\usepackage{amscd}
\usepackage{aligned-overset}
\usepackage{mathtools}
\usepackage{accents}
\usepackage{mathrsfs}
\usepackage{calc} 
\usepackage{ifthen} 

\hypersetup{
pdfauthor={Aday Celik, Mads Kyed},
pdftitle={Fluid-plate interaction under periodic forcing},
breaklinks=true,
colorlinks=true,
linkcolor=blue,
citecolor=blue,
urlcolor=blue,
filecolor=blue,
}

\numberwithin{equation}{section} 
\setkomafont{title}{\normalfont}

\newcommand{\eqrefsub}[2]{\eqref{#1}\textsubscript{#2}}

\newcommand{\ie}{i.e.}
\newcommand{\tin}{\text{in }}
\newcommand{\tif}{\text{if }}
\newcommand{\ton}{\text{on }}

\newcommand{\tand}{\text{and }}

\newcommand{\newCCtr}[2][d]{
\newcounter{#2}\setcounter{#2}{0}
\expandafter\xdef\csname kyedtheconst#2\endcsname{#1}}
\newcommand{\Cc}[2][nolabel]{
\stepcounter{#2}
\expandafter\ensuremath{\csname kyedtheconst#2\endcsname_{\arabic{#2}}}
\ifthenelse{\equal{#1}{nolabel}}{}
{\expandafter\xdef\csname kyedconst#1\endcsname
{\expandafter\ensuremath{\csname kyedtheconst#2\endcsname_{\arabic{#2}}}}}}
\newcommand{\Ccn}[2][nolabel]{
\expandafter\ensuremath{\csname kyedtheconst#2\endcsname}
\ifthenelse{\equal{#1}{nolabel}}{}
{\expandafter\xdef\csname kyedconst#1\endcsname
{\expandafter\ensuremath{\csname kyedtheconst#2\endcsname}}}}
\newcommand{\CcSetCtr}[2]{\setcounter{#1}{#2}}
\newcommand{\Cclast}[1]{\expandafter\ensuremath{\csname kyedtheconst#1\endcsname_{\arabic{#1}}}}
\newcommand{\Ccllast}[1]{
\addtocounter{#1}{-1}
\expandafter\ensuremath{\csname kyedtheconst#1\endcsname_{\arabic{#1}}}
\addtocounter{#1}{1}}
\newcommand{\const}[1]{\expandafter{\ifcsname kyedconst#1\endcsname\csname kyedconst#1\endcsname
\else\errmessage{Undefined Kyedconstant #1.}\fi}}
\newenvironment{pdeq}{ \left\{ \begin{aligned}}{\end{aligned}\right.}

\newcommand{\np}[1]{(#1)}
\newcommand{\bp}[1]{\big(#1\big)}
\newcommand{\Bp}[1]{\bigg(#1\bigg)}
\newcommand{\BBp}[1]{\Bigg(#1\Bigg)}
\newcommand{\bbp}{\BBp}
\newcommand{\nb}[1]{[#1]}
\newcommand{\bb}[1]{\big[#1\big]}
\newcommand{\bbb}[1]{\Big[#1\Big]}
\newcommand{\Bb}[1]{\bigg[#1\bigg]}
\newcommand{\BBb}[1]{\Bigg[#1\Bigg]}
\newcommand{\calp}{{\mathcal P}}

\newcommand{\calt}{{\mathcal T}}

\newcommand{\R}{\mathbb{R}}
\newcommand{\C}{\mathbb{C}}

\newcommand{\Z}{\mathbb{Z}}
\newcommand{\N}{\mathbb{N}}

\DeclareMathOperator{\e}{e}
\newcommand{\B}{B}

\DeclareMathOperator{\Div}{div}

\DeclareMathOperator{\supp}{supp}

\newcommand{\inverse}{{-1}}

\newcommand{\ra}{\rightarrow}

\newcommand{\set}[1]{\ensuremath{\{#1\}}}

\newcommand{\setc}[2]{\ensuremath{\{#1\,\lvert\,#2\}}}
\newcommand{\setcl}[2]{\ensuremath{\bigl\{#1\,\big\lvert\, #2\bigr\}}}

\newcommand{\proj}{\calp}

\newcommand{\quotientmap}{\pi}
\newcommand{\torus}{{\mathbb T}}

\newcommand{\transpose}{\top}

\newcommand{\idmatrix}{I}

\newcommand{\grad}{\nabla}

\newcommand{\gradp}{\nabla^\prime}

\newcommand{\pt}{\partial_t}

\newcommand{\dx}{{\mathrm d}x}

\newcommand{\ds}{{\mathrm d}s}
\newcommand{\dt}{{\mathrm d}t}

\newcommand{\dy}{{\mathrm d}y}

\newcommand{\ft}[1]{\widehat{#1}}

\newcommand{\FT}{\mathscr{F}}
\newcommand{\iFT}{\mathscr{F}^{-1}}

\newcommand{\f}{f}

\newcommand{\g}{g}
\newcommand{\norm}[1]{\lVert#1\rVert}
\newcommand{\normL}[1]{\Bigl\lVert#1\Bigr\rVert}
\newcommand{\snorm}[1]{{\lvert #1 \rvert}}

\newcommand{\snormL}[1]{{\Bigl\lvert #1 \Big\rvert}}

\newcommand{\CR}[1]{C^{#1}}

\newcommand{\WSR}[2]{\mathrm{W}^{#1,#2}} 
\newcommand{\WSRN}[2]{\mathrm{W}^{#1,#2}_0}

\newcommand{\LR}[1]{\mathrm{L}^{#1}}

\newcommand{\CRi}{\CR \infty}
\newcommand{\CRci}{\CR \infty_0}

\newcommand{\LRcompl}[1]{L^{#1}_{\bot}} 
 
\newcommand{\vvel}{v}
\newcommand{\vpres}{p}

\newcommand{\Vvel}{V}

\newcommand{\wvel}{w}
\newcommand{\wpres}{\pi}

\newcommand{\twpres}{\tilde{\pi}}

\newcommand{\uvel}{u}
\newcommand{\upres}{\mathfrak{p}}

\newcommand{\uvels}{\uvel_{\mathrm{s}}}
\newcommand{\uveltp}{\uvel_{\mathrm{tp}}}
\newcommand{\upress}{\upres_{\mathrm{s}}}
\newcommand{\uprestp}{\upres_{\mathrm{tp}}}

\newcommand{\Uvel}{U}

\newcommand{\Upres}{\mathfrak{P}}

\newcommand{\tuvel}{\widetilde{u}}
\newcommand{\tupres}{\widetilde{\mathfrak{p}}}

\newcommand{\tg}{\tilde{g}}

\newcommand{\fluidstress}{\mathrm T}

\newcommand{\half}{\frac{1}{2}}

\renewcommand{\epsilon}{\varepsilon}
\renewcommand{\Theta}{\varTheta}
\renewcommand{\phi}{\varphi}

\newcommand{\tay}{\calt}
\newcommand{\per}{\tay}

\newcommand{\fs}{f_{\mathrm{s}}}
\newcommand{\ftp}{f_{\mathrm{tp}}}

\newcommand{\F}{F}

\DeclareMathOperator{\OA}{\mathcal{S}}
\newcommand{\vertiii}[1]{{\left\vert\kern-0.25ex\left\vert\kern-0.25ex\left\vert #1 \right\vert\kern-0.25ex\right\vert\kern-0.25ex\right\vert}}

\newcommand{\sesqform}{B}
\newcommand{\TDN}{\operatorname{Tr}_{0}}

\newcommand{\Bogovskii}{\mathcal{B}}

\newcommand{\FPM}{\mathcal{F}}
\newcommand{\torusn}{\mathbb{T}_0}
\newcommand{\torusnn}{\torusn^2}
\newcommand{\torust}{\torus\times\torusnn}

\newcommand{\Et}{{\mathrm{E}}}
\newcommand{\EMt}[2]{{\mathrm{E}}_{#1, #2}}
\newcommand{\Eeta}{\Et\circ\phi_\eta}
\newcommand{\EMeta}[2]{\EMt{#1}{#2}\circ\phi_\eta}
\newcommand{\Omegaet}{\Omega_{\eta}^\torus}
\newcommand{\Omegap}{\Omega_+}
\newcommand{\Omegal}{\Omega}

\newcommand{\QT}{\mathrm{Q}_\torus}
\newcommand{\Deltap}{\Delta^\prime}
\newcommand{\Bilaplace}{\Delta^{\prime \, 2}}

\newcommand{\dxp}{ \, {\mathrm d}x'}

\newcommand{\PR}{\mathcal{P}}
\newcommand{\oPR}{\mathcal{P}_\bot}

\newcommand{\h}{h}
\newcommand{\etaft}{\widehat{\eta}}
\newcommand{\etak}{\eta_k}
\newcommand{\etas}{\eta_s}

\newcommand{\vvelk}{\vvel_k}

\newcommand{\uvelk}{\uvel_k}
\newcommand{\upresk}{\upres_k}
\newcommand{\tupress}{\widetilde{\upres}_s}
\newcommand{\tupresk}{\widetilde{\upres}_k}
\newcommand{\etatp}{\eta_{\mathrm{tp}}}

\newcommand{\WSRD}[2]{\mathrm{\dot{W}}^{#1,#2}}
\newcommand{\WSRNN}[2]{\mathrm{W}^{#1,#2}_{(0)}}
\newcommand{\WSRcompl}[2]{\mathrm{W}^{#1,#2}_{\bot}}

\newcommand{\CRcicompl}{\CR \infty_{0, \bot}}

\newcommand{\VSRK}{\mathcal{V}_k}
\newcommand{\ESR}[1]{\mathrm{S}^{#1}}

\newcommand{\SLS}[1]{\mathcal X^{#1}}
\newcommand{\SScompl}[1]{\mathcal X^{#1}_{\bot}}
\newcommand{\SSsigmacompl}[1]{\mathcal X^{#1}_{\bot}}
\newcommand{\DAS}[1]{\mathcal Y^{#1}}
\newcommand{\DScompl}[1]{\mathcal Y^{#1}_{\bot}}
\newcommand{\YFS}{\mathcal{Y}^q_{(0)}}
\newcommand{\nsnl}[1]{(#1\cdot\grad) #1}

\newcommand{\ff}{f}
\newcommand{\ffs}{f_s}
\newcommand{\ffk}{f_k}
\newcommand{\ffcur}{\ff_\eta}
\newcommand{\tff}{\tilde{\ff}}

\newcommand{\pert}{L}
\newcommand{\normalvec}{\nu}
\newcommand{\xprime}{x'}
\newcommand{\xp}{\xprime}
\newcommand{\xip}{\xi^\prime}

\newcommand{\tRS}{\mathrm{R}_\eta}
\newcommand{\tRF}{\mathrm{R}_{\f}}
\newcommand{\RFtilde}{\tilde{\mathrm{R}}_{\f}}
\newcommand{\tRD}{\tilde{\mathrm{R}}_{d}}
\newcommand{\RD}{\mathrm{R}_d}

\newcommand{\muf}{\mu_f}
\newcommand{\mus}{\mu_s}
\newcommand{\fp}{h}
\newcommand{\tfp}{\tilde{h}}
\newcommand{\fps}{{h_s}}
\newcommand{\fpk}{h_k}

\newcommand{\qdf}{q_0}

\newcommand{\Mt}{M_t}
\newcommand{\Mx}{M_x}
\newcommand{\mx}{m_x}
\newcommand{\fr}{\tilde{f}}
\newcommand{\frr}{\mathfrak{f}}
\newcommand{\gr}{\tilde{g}}
\newcommand{\Seta}{{\mathrm{S}}_\eta}
\newcommand{\varepsilonn}{\varepsilon_0}

\renewcommand{\dx}{ \, {\mathrm d}x}

\renewcommand{\dt}{ \, {\mathrm d}t}
\renewcommand{\dxp}{ \, {\mathrm d}x'}

\renewcommand{\upres}{p}

\renewcommand{\vpres}{\mathfrak p}

\theoremstyle{plain}
\newtheorem{thm}{Theorem}[section]

\newtheorem{lem}[thm]{Lemma}

\theoremstyle{remark}
\newtheorem{rem}[thm]{Remark}

\renewcommand{\Upres}{P}
\renewcommand{\tupres}{\widetilde{p}}

\newCCtr[C]{C}
\newCCtr[M]{M}
\newCCtr[B]{B}
\newCCtr[c]{c}
\newCCtr[\epsilon]{eps}
\CcSetCtr{eps}{-1}
\let\oldproof\proof
\def\proof{\CcSetCtr{c}{-1}\oldproof}

\begin{document}


\title{Fluid-plate interaction under periodic forcing}

\author{\textsc{Aday Celik}$^1$ and \textsc{Mads Kyed}$^{2,3}$\\[0.5em]
\textit{$^1$Fachbereich Mathematik}\\
\textit{Technische Universit\"at Darmstadt, Germany}\\
\href{mailto:celik@mathematik.tu-darmstadt.de}{celik@mathematik.tu-darmstadt.de}\\[0.5em]
\textit{$^2$Faculty of Science and Engineering}\\
\textit{Waseda University, Japan}\\
\href{mailto:kyed@aoni.waseda.jp}{kyed@aoni.waseda.jp}\\[0.5em]
\textit{$^3$Fachbereich Maschinenbau}\\
\textit{Flensburg University of Applied Sciences, Germany}\\
\href{mailto:mads.kyed@hs-flensburg.de}{mads.kyed@hs-flensburg.de}
}

\date{\today}
\maketitle

\begin{abstract}
The motion of a thin elastic plate interacting with a viscous fluid is investigated. 
A periodic force acting on the plate is considered, which in a setting without damping could lead to a resonant response.
The interaction with the viscous fluid provides a damping mechanism due to the energy dissipation in the fluid.
Moreover, an internal damping mechanism in the plate is introduced.
In this setting, we show that the periodic forcing leads to a time-periodic (non-resonant) solution.
We employ the Navier-Stokes and the Kirchhoff-Love plate equation
in a periodic cell structure to
model the motion of the viscous fluid and the elastic plate, respectively.
Maximal $\LR{q}$ regularity for the linearized system is established in a framework of time-periodic function spaces.
Existence of a solution to the fully nonlinear system is subsequently shown with a fixed-point argument.
\end{abstract}



\section{Introduction}

The interaction of an incompressible viscous fluid with a thin elastic structure is investigated. 
We consider a three-dimensional fluid filled container with an elastic plate as part of its boundary.
Our aim is to obtain a good understanding of the damping effects on the elastic structure in such a setting.
To this end, we introduce a periodic force, which in a setting without damping could lead to onset of resonance in the structure.
We develop a framework in which the damping effects under periodic forcing can be quantified, and employ the framework to show
how damping prevents resonance when the periodic force is sufficiently restricted in magnitude.

In order to simplify the mathematical analysis, we consider a simple geometry in which the container's stress free configuration is a cuboid with its bottom face an elastic plate.
We assume the motion of the viscous fluid is governed by the Navier-Stokes equations, and the motion of the plate by the Kirchoff-Love plate equations.
To further simplify the analysis, we impose periodic boundary conditions on the lateral faces of the cube.
In this setting we can study the linearized equations of motion directly in Fourier space, where we are able to quantify the damping effects in a precise manner.
From this characterization we are able to establish time-periodic a priori estimates, which we utilize to show that a solution to the nonlinear coupled fluid-structure system is necessarily time-periodic, that is, non-resonant, when the time-periodic force is sufficiently small.

In order to show the full mathematical potential of our technique, we have chosen to include an additional damping term in the structure equations that regularizes the system and enables us to establish comprehensive $\LR{q}$ estimates of maximal regularity type. Without this internal damping, our approach only yields a priori estimates in an $\LR{2}$ framework. We do not go further into this case in the current article. Further research is also required to extend our results to include more physically reasonable boundary conditions and geometries.

We denote by 
\begin{align*}
\omega:=(0,L)\times(0,L)\subset\R^2
\end{align*}
a square that represents a flat elastic plate in its stress free configuration, and consider the cuboid
\begin{align*}
\Omega\coloneqq\omega\times (0,1)\subset\R^3
\end{align*}
as container for the viscous fluid. More precisely, when no outer forces act on the system, the fluid filled cuboid container $\Omega$ with an elastic bottom face represents the stress free configuration of the fluid-structure system. As customary in fluid-structure problems, we utilize the stress free configuration  $\Omega$ as the \textit{reference configuration} and will refer to it as such.

Two $\per$-time-periodic outer body forces are introduced. A force 
\begin{align}\label{ForceFluid}
\ff:\R\times\Omega\ra\R^3, \quad \ff(t+\per,\cdot)=\ff(t,\cdot)
\end{align}
that acts on the fluid, and a force
\begin{align}\label{ForceStructure}
\fp:\R\times\omega\ra\R, \quad \fp(t+\per,\cdot)=\ff(t,\cdot)
\end{align}
that acts in normal direction on the plate (tangential forces are neglected in the Kirchoff-Love plate model). Here, $\R$ denotes the time axis. We consider generic forces formulated as functions on the reference configuration.

With outer forces acting on the system, the dynamics of the elastic structure is described by the displacement 
\begin{align*}
\eta:\R\times\omega\ra\R
\end{align*}
of the plate in normal direction $-e_3$. The \textit{current configuration} of the container at time $t$ is then
\begin{align*}
\Omega_\eta(t)\coloneqq \setc{x = \np{\xprime, x_3}\in\R^3}{x'\in\omega, \, -\eta(t,x') < x_3 < 1}.
\end{align*}
A canonical mapping that takes the reference configuration into the current configuration at time $t$ is given by
\begin{align}\label{Transformation}
\phi_\eta(t)\colon\Omega \to \Omega_\eta(t), \qquad x \mapsto \left( x', x_3 - \np{1-x_3}\eta(t,x')\right).
\end{align}
Provided $\eta$ is sufficiently small, $\phi_\eta(t)$ is a bijection with inverse given by
\begin{align*}
\phi_\eta^{-1}(t) \colon \Omega_\eta(t) \to \Omega, \qquad x \mapsto \left( x', \frac{x_3 + \eta(t,x')}{1 + \eta(t,\xprime)}\right).
\end{align*}

The dynamics of the incompressible viscous fluid flow is described in terms of its Eulerian velocity field and pressure term
\begin{align*}
\uvel: \bigcup_{t\in\R} \, \{t\}\times\Omega_{\eta}\np{t} \ra\R^3, \qquad \qquad
\upres: \bigcup_{t\in\R} \, \{t\}\times\Omega_{\eta}\np{t} \ra\R,
\end{align*}
respectively. We assume the fluid is Newtonian with Cauchy stress tensor given by 
\begin{align*}
\fluidstress(\uvel,\upres):=\muf\bp{\grad\uvel+\grad\uvel^\transpose}-\upres\idmatrix.
\end{align*}
Letting $\normalvec_t$ denote the outer normal of the container's current configuration $\Omega_\eta(t)$, 
we can thus express the surface force exerted by the fluid in normal direction $-e_3$ on the structure in reference configuration coordinates as
\begin{align*}
e_3\cdot\bp{\np{\fluidstress\np{\uvel, \upres}\normalvec_t}\circ\phi_\eta}_{| x_3= 0}.
\end{align*}

The Kirchoff-Love plate equation (with damping) govern the motion of the elastic structure:
\begin{align}\label{Plate}
\pt^2\eta + \Bilaplace\eta - \mus\Deltap\pt\eta &= \fp + e_3\cdot\bp{\np{\fluidstress\np{\uvel, \upres}\normalvec_t}\circ\phi_\eta}_{| x_3= 0}  \quad \tin\R\times\omega.
\end{align}
The term $\mus\Deltap\pt\eta$ with $\mus>0$ introduces an internal Kelvin-Voigt type damping in the structure. We use $\Deltap$ to denote the Laplacian with respect to the coordinates of the plate $\omega$.

The Navier-Stokes equations govern the motion of the incompressible viscous fluid:
\begin{align}\label{NSE}
\begin{pdeq}
\pt\uvel - \muf\Delta\uvel + \nsnl{\uvel} + \grad\upres &= \ffcur && \tin\bigcup_{t\in\R} \, \{t\}\times\Omega_{\eta}\np{t}, \\
\Div\uvel &= 0 && \tin\bigcup_{t\in\R} \, \{t\}\times\Omega_{\eta}\np{t}.
\end{pdeq}
\end{align}
Here, $\ffcur:=\ff\circ\phi_\eta^\inverse$ denotes the outer force $\ff$ expressed in current configuration coordinates, and
$\muf>0$ the kinematic viscosity constant. 

We assume a no-slip boundary condition for the fluid velocity on both the top and bottom (elastic) face of the container:
\begin{align}\label{BoundaryNormal}
\begin{aligned}
&\uvel  = 0 &&\ton\R\times\omega\times\{1\},\\
&\uvel\circ\phi_\eta = -\pt\eta\, e_3 &&\ton\R\times\omega\times\{0\}.
\end{aligned}
\end{align}
We assume periodic boundary conditions on the lateral faces of the container, that is, 
\begin{align}\label{BoundaryLateral}
\begin{aligned}
&\eta(\cdot,x_1+L,\cdot) = \eta(\cdot,x_1,\cdot), &&
\eta(\cdot,\cdot,x_2+L) = \eta(\cdot,\cdot,x_2),\\
&\uvel(\cdot,x_1+L,\cdot,\cdot) = \uvel(\cdot,x_1,\cdot,\cdot), &&
\uvel(\cdot,\cdot,x_2+L,\cdot) = \uvel(\cdot,\cdot,x_2,\cdot).
\end{aligned}
\end{align}
Moreover, we augment the system with the volume constraint
\begin{align}\label{VolConstraint}
\int_{\omega} \eta\np{t, \xprime} \dxp = 0.
\end{align}

We investigate the conditions under which the system \eqref{Plate}--\eqref{VolConstraint} admits a $\per$-time-periodic (non-resonant) solution under $\per$-time-periodic forcing
\eqref{ForceFluid}--\eqref{ForceStructure}. For this purpose, it is convenient to reformulate the system in a setting where the time-axis $\R$ is replaced with the torus $\torus = \R/\per\Z$. In such
a setting all function are generically $\per$-time-periodic. By the same token, we introduce the torus $\torusnn:=\np{\R/\pert\Z}^2$ and incorporate the periodic boundary conditions \eqref{BoundaryLateral} into the setting by replacing the elastic square $\omega$  and the cuboid container $\Omega$ with (we keep the notation $\omega$ and $\Omega$)
\begin{align*}
\omega:=\torusnn \qquad\tand \qquad \Omega:=\torusnn\times (0,1),
\end{align*}
respectively. We denote the corresponding time-space current configuration by
\begin{align*}
\Omegaet \coloneqq \bigcup_{t\in\torus} \, \{t\}\times\Omega_{\eta}\np{t}.
\end{align*}
Summarizing, we identify $\per$-time-periodic solutions $\np{\uvel, \upres, \eta}$ to \eqref{Plate}--\eqref{VolConstraint} as solutions to
\begin{align}\label{CoupledSystem}
\begin{pdeq}
\pt^2\eta + \Bilaplace\eta - \mus\Deltap\pt\eta &= \fp + e_3\cdot\bp{\np{\fluidstress\np{\uvel, \upres}\normalvec_t}\circ\phi_\eta}_{| x_3= 0} && \tin\torus\times\omega, \\
\pt\uvel - \muf\Delta\uvel + \nsnl{\uvel} + \grad\upres &= \ff\circ\phi_\eta^\inverse && \tin\Omegaet, \\
\Div\uvel &= 0 && \tin\Omegaet, \\
\uvel(t,x', -\eta(t,x')) &= -\pt\eta(t,x') e_3 && \ton\torus\times\omega, \\
\uvel(t,\xprime, 1) &= 0 && \ton\torus\times\omega,\\
\int_{\omega} \eta\np{t, \xprime} \dxp &= 0. &&
\end{pdeq}
\end{align}

As the main result of the article we establish existence of a solution to \eqref{CoupledSystem} under a smallness condition on the data:

\begin{thm}\label{mainThm}
Let $q\in (2,\infty)$. There is an $\varepsilon>0$ such that for all
\begin{align*}
\np{\ff, \fp} \in \LR{q}\np{\torus\times\Omega}^3 \times \LR{q}\np{\torus; \WSR{1-\frac{1}{q}}{q}\np\torusnn},
\end{align*}
satisfying 
\begin{align*}
\norm{\ff}_{\LR{q}\np{\torus\times\Omega}} + \norm{\fp}_{\LR{q}\np{\WSR{1-1/q}{q}\np\torusnn}} \leq \varepsilon
\end{align*}
there is a solution $\np{\eta,\uvel,\upres}$ to \eqref{CoupledSystem} satisfying
\begin{align*}
&\eta\in\WSR{2}{q}\bp{\torus; \WSR{1-\frac{1}{q}}{q}\np{\torusnn}}\cap \LR{q}\bp{\torus; \WSR{5-\frac{1}{q}}{q}\np{\torusnn}},\\
&\uvel\circ\phi_\eta\in \WSR{1}{q}\bp{\torus; \LR{q}\np\Omega}^3\cap\LR{q}\bp{\torus; \WSR{2}{q}\np\Omega}^3,\\
&\upres\circ\phi_\eta\in \LR{q}\np{\torus; \WSR{1}{q}\np{\Omega}}.
\end{align*}
\end{thm}

The main novelty of the article, however, concerns the technique we introduce to establish a priori estimates of the corresponding linearization:
\begin{align}\label{LinearizedSystem}
\begin{pdeq}
\pt^2\eta + \Bilaplace\eta - \mus\Deltap\pt\eta &= \fp -e_3\cdot\bp{\fluidstress(\uvel,\upres){e_3}}_{| x_3= 0} && \tin\torus\times\torusn^2, \\
\pt\uvel - \muf\Delta\uvel + \grad\upres &= \ff && \tin\torus\times\Omega, \\
\Div\uvel &= \g && \tin\torus\times\Omega, \\
\uvel\np{t, \xprime, 0} &= -\pt\eta\np{t, \xprime} e_3 && \ton\torus\times\torusn^2, \\
\uvel\np{t, \xprime, 1} &= 0 && \ton\torus\times\torusnn,\\
\int_{\torusnn} \eta\np{t, \xprime} \dxp &= 0. &&
\end{pdeq}
\end{align}
More specifically, \eqref{LinearizedSystem} is obtained as the linearization of \eqref{CoupledSystem} by reformulating the fluid equations 
\eqrefsub{CoupledSystem}{2-3} in the reference configuration and subsequently neglecting all higher order terms in $\np{\eta,\uvel,\upres}$. Moreover, an inhomogeneous right-hand side $g$ is introduced in
\eqrefsub{LinearizedSystem}{3}, which is critical in the utilization of \eqref{LinearizedSystem} towards the resolution of \eqref{CoupledSystem}.
We establish the following a priori $\per$-time-periodic estimates of \eqref{LinearizedSystem}:
\begin{thm}\label{MainThmFSLin}
Let $q\in(1,\infty)$. For all $\np{\ff, \g, \fp}\in\DAS{q}\np{\torus\times\Omega}$ with
\begin{align}\label{YSpaceDef}
\begin{aligned}
\DAS{q}\np{\torus\times\Omega}\coloneqq &\LR{q}\np{\torus\times\Omega}^3 \\
&\times \LR{q}\bp{\torus; \WSR{1}{q}\np\Omega}\cap\WSR{1}{q}\bp{\torus; \WSRD{-1}{q}\np\Omega}\\
& \times \LR{q}\bp{\torus; \WSR{1-\frac{1}{q}}{q}\np\torusnn}
\end{aligned}
\end{align}
satisfying
\begin{align*}
\int_\Omega \g \dx = 0,
\end{align*}
there is a unique solution $\np{\uvel, \upres, \eta}\in\SLS{q}\np{\torus\times\Omega}$ to \eqref{LinearizedSystem} with
\begin{align}\label{XSpaceDef}
\begin{aligned}
\SLS{q}\np{\torus\times\Omega}\coloneqq &\WSR{1}{q}\bp{\torus; \LR{q}\np\Omega}^3\cap\LR{q}\bp{\torus; \WSR{2}{q}\np\Omega}^3\\
&\times \LR{q}\np{\torus; \WSR{1}{q}\np{\Omega}}\\
&\times \WSR{2}{q}\bp{\torus; \WSR{1-\frac{1}{q}}{q}\np{\torusnn}}\cap \LR{q}\bp{\torus; \WSR{5-\frac{1}{q}}{q}\np{\torusnn}}.
\end{aligned}
\end{align}
Moreover,
\begin{align}\label{FSLinInhomEst}
\begin{aligned}
\norm{\np{\uvel, \upres, \eta}}_{\SLS{q}\np{\torus\times\Omega}} \leq \Cc{C}\,\norm{\np{\ff, \g, \fp}}_{\DAS{q}\np{\torus\times\Omega}}
\end{aligned}
\end{align}
with $\Cclast{C} = \Cclast{C}(q, \per)>0$.
\end{thm}

We solve \eqref{LinearizedSystem} via a representation formula in which the damping effect in the structure of both the fluid force $e_3\cdot\fluidstress(\uvel,\upres)e_3$ an the internal damping $\mus\Deltap\pt\eta$ are quantified in the Fourier space with respect to the Fourier transform $\FT_{\torus\times\torusn^2}$; see Remark \ref{DampingEffectInFourierSpace}. From this characterization we obtain the a priori estimate \eqref{FSLinInhomEst} via a transference principle
for Fourier multipliers.
We first establish the estimate in a half-space setting, and subsequently employ a (non-trivial) localization.

The coupling of an incompressible viscous fluid with an elastic plate has previously been investigated in
\cite{DEGT2000, GM2000, DE98, FO99, FSDesjardins05, VeigaFluidS, DS19, GaldiMahdi2014,VeigaFluidS,FSDesjardins05,ChueshovRyzhkova2013,GrandmontElasticPlate}.
Most of these articles cover the corresponding initial-value problem. 
To our knowledge, the investigation of time-periodic solutions in an $\LR{q}$ setting to the fully non-linear problem \eqref{CoupledSystem} is new.

\section{Preliminaries}

The periodic time-space domain $\torus\times\torusnn\times\R$ with $\torus := \R/\per\Z$ and $\torusnn:=\np{\R/\pert\Z}^2$ inherits its differentiable structure as a manifold from the physical time-space domain $\R\times\R^2\times\R$
via the quotient mappings $\quotientmap:\R\ra\torus$ and $\quotientmap_0:\R^2\ra\torusnn$. 
It is easy to verify that the reformulation \eqref{CoupledSystem} in the subdomain $\torus\times\torusnn\times (0,1)$ of $\torus\times\torusnn\times\R$ is equivalent to the original system \eqref{Plate}--\eqref{VolConstraint} defined in the physical time-space domain.
Sobolev spaces defined on subdomains of $\torus\times\torusnn\times\R$, such as those appearing in Theorem \ref{mainThm} and \ref{MainThmFSLin}, can be defined analogously to classical
Sobolev spaces. We refer to \cite{CelikKyed_StokesHalfSpace} for a systematic approach.

We take advantage of the structure of $\torus\times\torusnn$ as a compact abelian group (with normalized Haar measure) and utilize the corresponding Fourier transform $\FT_{\torus\times\torusnn}$.
We identify the dual group of $\torus\times\torusnn$ as $\frac{2\pi}{\per}\Z\times\np{\frac{2\pi}{\pert}\Z}^2$, and
use $(k,\xi)\in\frac{2\pi}{\per}\Z\times\np{\frac{2\pi}{\pert}\Z}^2$ as canonical notation for its elements. Formally, the Fourier transform takes the following form on functions
$\uvel:\torus\times\torusnn\ra\R$:
\begin{align*}
\ft{\uvel}(k,\xi) \coloneqq \FT_{\torus\times\torusnn}\nb{\uvel}(k,\xi) \coloneqq \int_\torus\int_{\torusnn} \uvel(t,x)\,\e^{-ix\cdot\xi-ik t}\,\dx\dt.
\end{align*}
The Fourier transform $\FT_{\torus\times\torusnn}$ can be expressed as the composition of the Fourier transforms $\FT_\torus$ and $\FT_{\torusnn}$, which takes functions defined on
$\torus$ and $\torusnn$ into their Fourier coefficients with respect to indices $k\in\frac{2\pi}{\per}\Z$ and $\xi\in\np{\frac{2\pi}{\pert}\Z}^2$.
At one point in the following, namely when we establish the $\LR{q}$ estimates in Theorem \ref{MainThmFSLin}, it is critical that a single Fourier transform $\FT_{\torus\times\torusnn}$ covering the whole time-space domain is employed instead of a sequential utilization of $\FT_\torus$ and $\FT_{\torusnn}$.

For functions $f$ defined on $\torus\times\torusnn\times\R$ we define
\begin{equation}\label{Projections}
\fs\coloneqq\PR f (t,\cdot) := \int_\torus f(s,\cdot)\,\ds,\quad \ftp\coloneqq\oPR f(t,\cdot):= f(t,\cdot)-\proj f(t,\cdot)
\end{equation}
whenever the integral is well defined. Since $\fs$ is independent of time $t$, we shall implicitly treat $\fs$ as a function 
in the spatial variable $x$ only and refer to it as the \emph{steady-state}  part of $\f$.
The function $\ftp$ is referred to as the \emph{purely oscillatory} part of $\f$.
When using the projections $\PR$ and $\oPR$ to decompose a function space, we use the symbol $\bot$ as subscript to denote the purely oscillatory part, for example
$\LRcompl{q}\np{\torus\times\Omega} := \oPR\LR{q}\np{\torus\times\Omega}$.

Finally, we observe that the divergence problem
\begin{align}\label{DivProblem}
\begin{pdeq}
\Div\uvel &= \f && \tin\Omega, \\
\uvel &= 0 && \ton\partial\Omega
\end{pdeq}
\end{align}
set in the torus domain $\Omega\coloneqq \torusnn\times (0,1)$ possesses the same properties as the corresponding divergence problem set in classical subdomains of $\R^3$. 
As in \cite[Section III.3]{Galdi}, a so-called Bogovski\u{\i} operator $\Bogovskii\colon\CRci\np\Omega \to \CRci\np\Omega^3$ can be constructed such that $\uvel:=\Bogovskii(f)$ satisfies \eqref{DivProblem} whenever 
$\f$ satisfies
\begin{align}\label{BogovskiiMV}
\int_\Omega \f \dx = 0.
\end{align}

\begin{thm}[Bogovski\u{\i} Operator]\label{BogovskiiThm}
The Bogovski\u{\i} operator $\Bogovskii\colon\CRci\np\Omega\to\CRci\np\Omega^3$ has
a continuous (linear) extension $\Bogovskii\colon\WSRN{m}{q}\np\Omega\to\WSRN{m+1}{q}\np\Omega^3$ for $q\in (1,\infty)$ and $m\in\N_0$.
If $\f\in\WSRN{m}{q}\np\Omega$ satisfies
\eqref{BogovskiiMV}, then $\uvel := \Bogovskii\f$ is a solution
to \eqref{DivProblem} satisfying
\begin{align}\label{BogovskiiEst1}
\norm{\grad\uvel}_{l,q}\leq\Cc{C}\norm\f_{l,q},
\end{align}
for all $l = 0, \ldots, m$. Moreover, there exists a constant
$\Cc{C} = \Cclast{C}\np{n, q, \Omega}>0$ such that
\begin{align}\label{BogovskiiLqEst}
\norm{\uvel}_{\LR{q}\np\Omega} \leq \Ccn{C}\snorm{\f}_{-1, q}^*,
\end{align}
where 
\begin{align}\label{BogovskiiSeminorm}
\snorm{\f}_{-1, q}^* = \sup_{\phi\in\WSRD{1}{q'}\np\Omega; \; \snorm{\phi}_{1, q'} = 1} \snorm{(\f, \phi)}
\end{align}
for all $\f\in\LR{q}\np\Omega$.
\end{thm}

\begin{proof}
The classical construction of the Bogovski\u{\i} operator can be adapted to the domain $\Omega$ without significant modifications; see for example \cite[Section III.3]{Galdi}.  
The first part of the theorem can be shown as in \cite[Theorem III.3.3]{Galdi}.
Estimate \eqref{BogovskiiLqEst} follows as in the proof of \cite[Theorem III.3.5]{Galdi}.
\end{proof}

\section{Linearized system}

We employ the projections $\PR$ and $\oPR$ introduced in \eqref{Projections} to decompose \eqref{LinearizedSystem} into a steady-state part and a purely oscillatory part.
These two problems are different by nature, and we therefore study them separately. The purely oscillatory part is investigated in Section \ref{ResolvenProblemSection}--\ref{ExisteceUniquenessPPSection}, and
the steady-state problem in Section \ref{SteadyStateSection}. A proof of Theorem \ref{MainThmFSLin} is presented in Section \ref{ProofOfLinMainThmSection}.
To simplify the notation, we set, without loss of generality,  $\mus=\muf=1$.

\subsection{Resolvent problem}\label{ResolvenProblemSection}

Applying the Fourier transform $\FT_\torus$ to the linear system \eqref{LinearizedSystem}, we obtain for each $k\in\frac{2\pi}{\per}\Z$ the following resolvent type system
for the (complex valued) Fourier coefficients $\np{\etak,\uvelk,\upresk}$:
\begin{align}\label{FSLinHomDataResolvent}
\begin{pdeq}
-k^2\etak + \Bilaplace\etak - ik\Deltap\etak &= \fpk - e_3\cdot\bp{\fluidstress({\uvelk},{\upresk})e_3}_{| x_3= 0} && \tin\torusn^2, \\
ik\uvelk - \Delta\uvelk + \grad\upresk &= \ffk && \tin\Omega, \\
\Div\uvelk &= \g_k && \tin\Omega, \\
\uvel_{k |x_3=0} &= -ik\etak e_3 && \ton\torusn^2, \\
\uvel_{k |x_3=1} &= 0 && \ton\torusnn,\\
\int_{\torusnn} \etak\np{\xprime} \dxp &= 0. &&
\end{pdeq}
\end{align}

In the homogeneous case $\g_k=0$, we can solve this system with an application of Lax-Milgram's theorem:

\begin{lem}[Existence]\label{ResolventLaxMilgramLem}
Let $k\in \frac{2\pi}{\per}\Z$ with $k\neq 0$. For every
\begin{align*}
(\ffk,\fpk)\in\LR{2}\np{\Omega;\C}^3\times\LR{2}\np{\torusnn;\C}
\end{align*}
there is a weak solution $(\uvelk,\etak)$ to \eqref{FSLinHomDataResolvent} with $\g_k=0$, that is,
$(\etak,\vvelk)\in\VSRK$ with
\begin{align}
\begin{aligned}
\VSRK := \setcl{(\uvel,\eta)&\in\WSR{1}{2}\np{\Omega;\C}^3\times\WSR{2}{2}(\torusnn;\C)}{\\
\qquad \qquad &\int_{\torusnn} \eta \np{\xprime} \dxp = 0,\ \ 
\uvel_{|x_3=0} = -ik\eta e_3\ \ton\torusn^2,\ \ 
\uvel_{|x_3=1} = 0\ \ton\torusnn}
\label{ResolventLaxMilgramLemVspace}
\end{aligned}
\end{align}
and satisfies
\begin{align}\label{ResolventLaxMilgramLemCond}
\forall\np{\wvel,\zeta}\in\VSRK: \quad \sesqform\bp{\np{\uvelk,\etak},\np{\wvel,\zeta}} = -ik \int_{\torusnn}\fpk \overline{\zeta}\,\dxp + \int_{\Omega} \ffk\cdot\overline{\wvel} \dx
\end{align}
where
\begin{align}\label{ResolventLaxMilgramLemSesqform}
\begin{aligned}
\sesqform\bp{\np{\uvel,\eta},\np{\wvel,\zeta}} := &\int_{\torusnn} ik^3\eta \overline{\zeta} - ik\Deltap \eta \Deltap\overline{\zeta} + k^2\gradp\eta\cdot\gradp\overline{\zeta}\,\dxp \\
& +  \int_\Omega \grad\uvel:\grad\overline\wvel + ik \uvel\cdot\overline\wvel\,\dx.
\end{aligned}
\end{align}
Moreover,
\begin{align}\label{ResolventLaxMilgramEst}
\norm{\uvelk}_{1,2} + \norm{\etak}_{2,2} + \norm{k \etak}_{1,2} \leq \Cc{C} \bp{ \norm{\ffk}_2 + \norm{\fpk}_2}
\end{align}
with constant $\Cclast{C}=\Cclast{C}\np{\per, \pert}>0$ independent on $k$.
\end{lem}

\begin{proof}
The sesquilinear form $\sesqform$ is clearly bounded in the Hilbert space $\VSRK$.
By computing the real and imaginary part of $\sesqform\bp{\np{\uvel,\eta},\np{\uvel,\eta}}$, one readily verifies that $\sesqform$ is also coercive. Existence of a solution to \eqref{ResolventLaxMilgramLemCond} therefore follows from the theorem of Lax-Milgram. Setting $\np{\wvel,\zeta}=\np{\uvelk,\etak}$ in \eqref{ResolventLaxMilgramLemCond} and taking real and imaginary parts in the equation, one obtains \eqref{ResolventLaxMilgramEst} by an application of Young's inequality.
\end{proof}

A pressure field corresponding to the weak solution obtained in Lemma \ref{ResolventLaxMilgramLem} can be constructed and higher order regularity subsequently established.

\begin{lem}[Pressure and regularity]\label{ExistencePressure}
Let $k\in\frac{2\pi}{\per}\Z\setminus\{0\}$ and  $r\in\N_0$. Moreover, let
$\np{\ffk,\fpk}\in\WSR{r}{2}\np{\Omega;\C}^3\times\WSR{r}{2}\np{\torusnn;\C}$ and
$\np{\uvelk, \etak}\in{\VSRK}$ be the corresponding weak solution
to \eqref{FSLinHomDataResolvent} with $\g_k=0$ constructed in
Lemma \ref{ResolventLaxMilgramLem}. Then there exists a pressure field $\upresk$ such that
\begin{align}\label{ExistencePressureRegularity}
\np{\uvelk, \upresk, \etak}\in\WSR{r+2}{2}\np\Omega^3\times\WSR{r+1}{2}\np\Omega\times\WSR{r+4}{2}\np\torusnn
\end{align}
solves \eqref{FSLinHomDataResolvent} and satisfies
\begin{align}\label{FSResolventEnergyEst}
\norm{\uvelk}_{\WSR{r+2}{2}\np{\Omega}} &+ \norm{\upresk}_{\WSR{r+1}{2}\np\Omega} + \norm{\etak}_{\WSR{r+4}{2}\np{\torusnn}} \leq \Cc{C}\bp{\norm{\ffk}_{\WSR{r}{2}\np\Omega} + \norm{\fpk}_{\WSR{r}{2}\np\torusnn}}
\end{align}
with $\Cclast{C} = \Cclast{C}\np{\per, \pert}>0$ independent of $k$.
\end{lem}

\begin{proof}
First consider only the Stokes part \eqrefsub{FSLinHomDataResolvent}{2-5} of the system. By well known methods (see for example Theorem 1.2. in \cite{FarwigSohr93})
a pressure $\tupresk\in\LR{2}\np\Omega$ can be constructed such that $(\uvelk,\tupresk)$ solves this resolvent type Stokes problem.
Since the weak formulation \eqref{ResolventLaxMilgramLemCond} is obtained by multiplying \eqref{FSLinHomDataResolvent} with a pair of test functions and subsequent integration by parts,
one readily verifies that $\np{\uvelk, \upresk, \etak}$ with
\begin{align*}
\upresk := {\tupresk} - \int_{\torusnn} {\tupresk} \dxp - \int_{\torusnn} \fpk \dxp
\end{align*}
solves the full system \eqref{FSLinHomDataResolvent} in a distributional sense.

At the outset $\etak\in\WSR{2}{2}\np{\torusnn}$. Standard elliptic regularity theory for the Stokes system \eqrefsub{FSLinHomDataResolvent}{2-5} (see for example \cite[Chapter 4]{Galdi}) therefore yields
$(\uvelk,\upresk)\in\WSR{2}{2}(\Omega)^3\times\WSR{1}{2}(\Omega)$.
Consequently, $\fluidstress({\uvelk},{\upresk})\in\WSR{\half}{2}(\torusnn)$. Applying the Fourier transform $\FT_{\torusnn}$ to the plate equation
\eqrefsub{FSLinHomDataResolvent}{1}, we obtain the representation
\begin{align}\label{PlateRepThm}
\etak = \iFT_{\torusnn}\Bb{\frac{1}{\snorm{\xip}^4 - k^2 + ik\snorm{\xip}^2}\FT_{\torusnn}\nb{e_3\cdot\bp{\fluidstress({\uvelk},{\upresk})e_3}_{|x_3 = 0} + \fpk}}.
\end{align}
Due to the regularizing damping term $ik\snorm{\xip}^2$ in the Fourier multiplier, a simple application of Parseval's theorem implies that $\etak$ admits a regularity gain of 4 derivatives over the right-hand side of the (damped) plate equation \eqrefsub{FSLinHomDataResolvent}{1}. Consequently we deduce $\etak\in\WSR{\frac{5}{2}}{2}\np{\torusnn}$. Iterating this procedure, we find that
$\np{\uvelk, \upresk, \etak}$ is as regular as the data $\np{\ffk,\fpk}$ allows for and thus conclude \eqref{ExistencePressureRegularity}.
In this process, the elliptic regularity theory of the Stokes system and the representation \eqref{PlateRepThm} also yields \eqref{FSResolventEnergyEst}.
\end{proof}

\subsection{Purely oscillatory problem}\label{ExisteceUniquenessPPSection}

By expressing a solution to \eqref{LinearizedSystem} in terms of its Fourier series, we can utilize the  existence and regularity of the Fourier coefficients in Lemma \ref{ResolventLaxMilgramLem} and Lemma \ref{ExistencePressure} to construct a solution to \eqref{LinearizedSystem}:

\begin{lem}[Existence]\label{EsistenceLinTPProblem}
Let $q\in(1,\infty)$. For any $\np{\ff,\fp}\in\CRcicompl\np{\torus\times\Omega}^3\times\CRcicompl\np{\torust}$
the system \eqref{LinearizedSystem} with $g=0$ admits a solution
$\np{\uvel, \upres, \eta}\in\SScompl{q}\np{\torus\times\Omega}$.
\end{lem}

\begin{proof}
We expand the data into Fourier series with respect to $\torus$.
Since $\PR\ff=\PR\fp=0$, the zeroth order Fourier coefficients are 0, that is,
\begin{align*}
\ff = \sum_{k\in\frac{2\pi}{\per}\Z\setminus\set{0}} \ffk e^{ikt} \qquad\text{and}\qquad \fp = \sum_{k\in\frac{2\pi}{\per}\Z\setminus\set{0}} \fpk e^{ikt}
\end{align*}
with Fourier coefficients $\ffk=\FT_\torus\nb{\ff}(k)$ and $\fpk=\FT_\torus\nb{\fp}(k)$.
By Parseval's theorem, these identities are valid in $\WSR{r}{2}\bp{\torus;\WSR{2}{r}\np{\Omega}}$ and $\WSR{r}{2}\bp{\torus;\WSR{2}{r}\np{\torusnn}}$, respectively, for any $r\in\N_0$.
From Lemma \ref{ResolventLaxMilgramLem} and Lemma \ref{ExistencePressure} we obtain for each pair $\np{\ffk,\fpk}$ of Fourier coefficients a solution
\begin{align*}
\np{\uvelk, \upresk, \etak}\in\WSR{r+2}{2}\np\Omega^3\times\WSR{r+1}{2}\np\Omega\times\WSR{r+4}{2}\np\torusnn
\end{align*}
to \eqref{FSLinHomDataResolvent}. By \eqref{FSResolventEnergyEst} and Parseval's theorem, the corresponding Fourier series
\begin{align*}
\uvel\coloneqq \sum_{k\in\frac{2\pi}{\per}\Z\setminus\{0\}} \uvelk e^{ikt}, \qquad \upres\coloneqq \sum_{k\in\frac{2\pi}{\per}\Z\setminus\{0\}} \upresk e^{ikt}, \qquad \eta\coloneqq \sum_{k\in\frac{2\pi}{\per}\Z\setminus\{0\}} \etak e^{ikt}
\end{align*}
are well-defined in the Hilbert spaces
$\WSR{r}{2}\bp{\torus;\WSR{r+2}{2}\np\Omega}^3 $, $\WSR{r}{2}\bp{\torus;\WSR{r+1}{2}\np\Omega}$ and $\WSR{r}{2}\bp{\torus;\WSR{r+4}{2}\np\torusnn}$, respectively.
Clearly, $\np{\uvel, \upres, \eta}$ solves \eqref{LinearizedSystem}. 
Choosing $r$ sufficiently large, we obtain $\np{\uvel, \upres, \eta}\in\SScompl{q}\np{\torus\times\Omega}$ by Sobolev embedding.  
\end{proof}

The Fourier analysis carried out above is rather crude, and it is restricted to the Hilbert space setting due to the application of Parseval's theorem. It is, however, only a step towards a more refined Fourier analysis that leads to maximal regularity $\LR{q}$ estimates. To this end, we need the following uniqueness property:

\begin{lem}[Uniqueness]\label{FSLinUniquenessLem}
Let $q\in(1,\infty)$.
A solution to 
\eqref{LinearizedSystem} is unique in the class $\SLS{q}\np{\torus\times\Omegal}$.
\end{lem}

\begin{proof}
It suffices to consider a solution $\np{\uvel, \upres, \eta}\in\SLS{q}\np{\torus\times\Omegal}$ to \eqref{LinearizedSystem} with homogeneous right-hand side $\np{\ff, \g, \fp}=(0,0,0)$ and show that necessarily $\np{\uvel, \upres, \eta}=(0,0,0)$.
We employ a duality argument. Let $(\phi,\psi)\in\CRcicompl\np{\torus\times\Omega}^3\times\CRcicompl\np{\torust}$.
By the same argument that leads to Lemma \ref{EsistenceLinTPProblem}, existence of a solution $(\wvel,\wpres,\zeta)\in\SScompl{q'}\np{\torus\times\Omega}$ to the dual of
\eqref{LinearizedSystem} follows, that is,
\begin{align}\label{LinearizedSystemDual}
\begin{pdeq}
\pt^2\zeta + \Bilaplace\zeta + \Deltap\pt\zeta &= \psi -e_3\cdot\bp{\fluidstress(\wvel,\wpres)e_3}_{| x_3= 0} && \tin\torus\times\torusn^2, \\
-\pt\wvel - \Delta\wvel - \grad\wpres &= \phi && \tin\torus\times\Omega, \\
\Div\wvel &= 0 && \tin\torus\times\Omega, \\
\wvel\np{t, \xprime, 0} &= -\pt\zeta\np{t, \xprime} e_3 && \ton\torus\times\torusn^2, \\
\wvel\np{t, \xprime, 1} &= 0 && \ton\torus\times\torusnn,\\
\int_{\torusnn} \zeta\np{t, \xprime} \dxp &= 0. &&
\end{pdeq}
\end{align}
An straightforward integration by parts shows that
\begin{align*}
\int_{\torus\times\Omega} \uvel \cdot \phi\,\dx\dt = \int_{\torusnn} \partial_t\eta\,\psi\,\dxp.
\end{align*}
Since $(\phi,\psi)$ can be chosen arbitrarily, $\uvel=0$ and $\partial_t\eta=0$ follows. Since $\int_{\torusnn}\eta\dxp = 0$ we deduce $\eta=0$ and in turn $\upres=0$.
\end{proof}

Finally we can move towards $\LR{q}$ estimates. We first consider the problem \eqref{LinearizedSystem} in 
the periodic half space $\Omegap = \torusnn\times\R_+$. In this geometry we can compute a formula that represents the solution $\eta$ in terms of the data.

\begin{lem}[$\LR{q}$-Estimate in $\Omegap$]\label{POPFSExistence}
Let $q\in(1,\infty)$. For any $\np{\f,\g,\h}\in\DScompl{q}\np{\torus\times\Omegap}$
a solution $\np{\vvel, \vpres, \eta}\in\SScompl{q}\np{\torus\times\Omegap}$ to
\begin{align}\label{FullyInhomLinearizedSystemtem}
\begin{pdeq}
\pt^2\eta + \Bilaplace\eta - \Deltap\pt\eta &= \h-e_3\cdot\bp{\fluidstress(\vvel,\vpres)e_3}_{| x_3= 0} && \tin\torus\times\torusn^2, \\
\pt\vvel - \Delta\vvel + \grad\vpres &= \f && \tin\torus\times\Omegap, \\
\Div\vvel &= \g && \tin\torus\times\Omegap, \\
\vvel_{|x_3=0} &= -\pt\eta e_3 && \ton\torus\times\torusn^2,\\
\int_{\torusnn} \eta(t,\xp) \dxp &= 0
\end{pdeq}
\end{align}
obeys the $\LR{q}$ estimate
\begin{align}\label{POPFSEst}
\begin{aligned}
&\norm{\vvel}_{\WSR{1,2}{q}\np{\torus\times\Omegap}} + \norm{\grad\vpres}_{\LR{q}\np{\torus\times\Omegap}} +
\norm{\eta}_{\WSR{2}{q}\np{\torus; \WSR{1-\frac{1}{q}}{q}\np{\torusnn}}\cap \LR{q}\np{\torus; \WSR{5-\frac{1}{q}}{q}\np{\torusnn}}}\\
& \quad \leq \Cc{C}\bp{\norm\f_{\LR{q}\np{\torus\times\Omegap}} + \norm\g_{\LR{q}\np{\torus; \WSR{1}{q}\np\Omegap}\cap\WSR{1}{q}\np{\torus; \WSRD{-1}{q}\np\Omegap}} + \norm\h_{\LR{q}\np{\torus; \WSR{1-\frac{1}{q}}{q}\np\torusnn}}}.
\end{aligned}
\end{align}
\end{lem}

\begin{proof}
From \cite{CelikKyed_StokesHalfSpace} (see also \cite[Theorem 4.4.7]{CelikDiss}) we obtain a solution
\begin{align*}
\np{\wvel, \wpres}\in\WSRcompl{1,2}{q}\np{\torus\times\Omegap}^3\times\LRcompl{q}\bp{\torus; \WSRD{1}{q}\np\Omegap}
\end{align*}
to the half-space Stokes problem
\begin{align}\label{FSStokesUniqueness}
\pt\wvel - \Delta\wvel + \grad\wpres = \f, \quad \Div\wvel = \g , \quad \wvel_{|\torus\times\partial\Omegap} = 0
\end{align}
satisfying
\begin{align}\label{APrioriLqEstPP}
\begin{aligned}
&\norm{\wvel}_{\WSR{1,2}{q}\np{\torus\times\Omegap}} + \norm{\grad\wpres}_{\LR{q}\np{\torus\times\Omegap}} \\
&\qquad\qquad\qquad\qquad\quad \leq \Cc{c}\bp{\norm\f_{\LR{q}\np{\torus\times\Omegap}} + \norm\g_{\LR{q}\np{\torus; \WSR{1}{q}\np\Omegap}\cap\WSR{1}{q}\np{\torus; \WSRD{-1}{q}\np\Omegap}}}.
\end{aligned}
\end{align}
Letting
\begin{align*}
\uvel:= \vvel - \wvel, \qquad  \upres := \vpres - \wpres + \int_{\Omega} \wpres \dy,                    
\end{align*}
we find that $\np{\uvel,\upres,\eta}$ solves
\begin{align}\label{LRqInhomLinearizedSystemtem}
\begin{pdeq}
\pt^2\eta + \Bilaplace\eta - \Deltap\pt\eta &= \F -e_3\cdot\bp{\fluidstress(\uvel,\upres)e_3}_{| x_3= 0} && \tin\torus\times\torusn^2, \\
\pt\uvel - \Delta\uvel + \grad\upres &= 0 && \tin\torus\times\Omegap, \\
\Div\uvel &= 0 && \tin\torus\times\Omegap, \\
\uvel_{|x_3=0} &= -\pt\eta\, e_3 && \ton\torus\times\torusn^2,\\
\int_{\torusnn} \eta(t,\xp) \dxp &= 0
\end{pdeq}
\end{align}
with 
\begin{align*}
\F := \h -e_3\cdot\fluidstress\bp{\wvel,\twpres}e_3, \qquad \twpres:=\wpres-\int_{\Omega} \wpres\dx.
\end{align*}
At this point we mimic the arguments in
\cite[Proof of Proposition 3.1]{CelikKyed_StokesHalfSpace} and utilize the Fourier transform $\FT_{\torust}$ in
\eqref{LRqInhomLinearizedSystemtem}. The result is a system of ODEs we can solve explicitly obtaining 
\begin{align}\label{POP_FS_SolFormulaOhneEta}
\uvel = (U,V), \qquad \upres = \iFT_{\torus\times\torusnn}\bb{\qdf(k,\xip)\, e^{-\snorm{\xip}x_3}}
\end{align}
with 
\begin{align}\label{POP_FS_SolFormulaAuxOhneEta}
\begin{aligned}
\Uvel &:= \iFT_{\torus\times\torusnn}\Bb{-\frac{\xip\qdf}{k} e^{-\snorm{\xip}x_3} + \frac{\xip\qdf}{k} \e^{-\sqrt{\snorm{\xip}^2 + ik} \, x_3}}, \\
\Vvel &:= \iFT_{\torus\times\torusnn}\Bb{\frac{\snorm{\xip}\qdf}{ik} e^{-\snorm{\xip}x_3} - \bp{ik\etaft + \frac{\snorm{\xip}\qdf}{ik}} 
e^{-\sqrt{\snorm{\xip}^2 + ik} \, x_3}}, \\
\qdf(k,\xip) &:= \bbb{-ik\bp{\snorm{\xip} + \sqrt{\snorm{\xip}^2 + ik}} + \frac{k^2}{\snorm\xip}}\etaft.
\end{aligned}
\end{align}
Observing that $-e_3\cdot\fluidstress(\uvel,\upres)e_3=\upres_{|x_3=0}$ for $\uvel$ satisfying \eqrefsub{LRqInhomLinearizedSystemtem}{3-4}, we deduce from \eqrefsub{LRqInhomLinearizedSystemtem}{1} that
\begin{align}\label{etaRepFormulaPre1}
\Bb{\snorm{\xip}^4 - k^2 + ik\snorm{\xip}^2  - \frac{k^2}{\snorm{\xip}} + ik\Bp{\snorm{\xip} + \sqrt{\snorm{\xip}^2 + ik}}}\etaft &= \widehat{\F}.
\end{align}
We let $\delta_{\frac{2\pi}{\per}\Z}$ denote the Dirac measure on the group $\frac{2\pi}{\per}\Z$, that is, the function that takes the value 1 for $k=0$ and otherwise 0. 
By assumption $\PR\eta=0$, which means that the zeroth order Fourier coefficient of $\eta$ with respect to $\FT_{\torus}$ vanishes and consequently
\begin{align}\label{etaRepFormulaPre2}
\bp{1 - \delta_{\frac{2\pi}{\per}\Z}(k)}\, \etaft = \etaft.
\end{align}
Similarly, we let $\delta_{\np{\frac{2\pi}{\pert}\Z}^2}$ denote the Dirac measure on the group $\np{\frac{2\pi}{\pert}\Z}^2$, and conclude from \eqrefsub{LRqInhomLinearizedSystemtem}{5} that
\begin{align}\label{etaRepFormulaPre3}
\bp{1-\delta_{\np{\frac{2\pi}{\pert}\Z}^2}(\xip)}\, \etaft = \etaft.
\end{align}
We derive from \eqref{etaRepFormulaPre1}--\eqref{etaRepFormulaPre3} that
\begin{align}\label{POP_FS_SolFormulaMitEta}
\begin{aligned}
\eta &= \iFT_{\torus\times\torusnn}\BBb{
\frac{\bp{1 - \delta_{\frac{2\pi}{\per}\Z}(k)} \bp{1-\delta_{\np{\frac{2\pi}{\pert}\Z}^2}(\xip)}}{
\snorm{\xip}^4 - k^2 + ik\snorm{\xip}^2  - \frac{k^2}{\snorm{\xip}} + ik\Bp{\snorm{\xip} + \sqrt{\snorm{\xip}^2 + ik}}}
\widehat{\F}}.
\end{aligned}
\end{align}
Following the exact same steps as in \cite[Proof of Proposition 3.1]{CelikKyed_StokesHalfSpace}, a Fourier multiplier argument based on the representation formulas in \eqref{POP_FS_SolFormulaAuxOhneEta} yields
\begin{align*}
\begin{aligned}
\norm{\uvel}_{\WSR{1,2}{q}\np{\torus\times\Omegap}} + \norm{\grad\upres}_{\LR{q}\np{\torus\times\Omegap}}  &\leq \Cc{c}\,\bp{\norm{\pt\eta}_{\WSR{1-\frac{1}{2q}, 2-\frac{1}{q}}{q}\np{\torus\times\torusnn}}+\norm{\pt\eta}_{\WSR{1}{q}\np{\torus; \WSRD{-\frac{1}{q}}{q}(\torusnn)}}} \\
&\leq \Cc{c} \norm{\eta}_{\WSR{2}{q}\np{\torus; \WSR{1-\frac{1}{q}}{q}\np{\torusnn}}\cap \LR{q}\np{\torus; \WSR{5-\frac{1}{q}}{q}\np{\torusnn}}}.
\end{aligned}
\end{align*}
A similar argument further yields an estimate of $\eta$ based on the representation formula \eqref{POP_FS_SolFormulaMitEta}
and an analysis of the Fourier multiplier $M:\frac{2\pi}{\per}\Z\times \np{\frac{2\pi}{\pert}\Z}^2 \ra \C$ given by
\begin{align*}
M(k,\xip):= \frac{\bp{1 - \delta_{\frac{2\pi}{\per}\Z}(k)} \bp{1-\delta_{\np{\frac{2\pi}{\pert}\Z}^2}(\xip)}}{
\snorm{\xip}^4 - k^2 + ik\snorm{\xip}^2  - \frac{k^2}{\snorm{\xip}} + ik\Bp{\snorm{\xip} + \sqrt{\snorm{\xip}^2 + ik}}}.
\end{align*}
Observe that
\begin{align}\label{fundamultplier}
\np{k,\xip}\mapsto \bp{1+ \snorm{k}^2 + \snorm\xip^4}\,M(k,\xip)
\end{align}
is bounded. In fact, one can verify that a canonical extension of the multiplier in \eqref{fundamultplier} from the domain
$\frac{2\pi}{\per}\Z\times \np{\frac{2\pi}{\pert}\Z}^2$ to the domain $\R\times\R^2$ satisfies 
the conditions of Marcinkiewicz's multiplier theorem; see \cite[Lemma A.2.3 and Lemma A.2.4]{CelikDiss}.
It follows from de \textsc{De Leeuw}'s transference principle \cite{Leeuw1965} in combination with
Marcinkiewicz's multiplier theorem (see \cite{KyedSauer_Heat} for a comprehensive explanation of the argument) that 
\begin{align*}
\norm{\eta}_{\WSR{2}{q}\np{\torus; \WSR{1-\frac{1}{q}}{q}\np{\torusnn}}\cap \LR{q}\np{\torus; \WSR{5-\frac{1}{q}}{q}\np{\torusnn}}} &=
\norm{\iFT_{\torust}\bb{\np{1 + \snorm{k}^2 + \snorm\xip^4}\etaft}}_{\LR{q}\np{\torus; \WSR{1-\frac{1}{q}}{q}\np\torusnn}} \\
& \quad \leq \Cc{c}\norm{\F}_{\LR{q}\np{\torus; \WSR{1-\frac{1}{q}}{q}\np\torusnn}}.
\end{align*}
Since $\tilde{\wpres}$ has a vanishing mean value, we obtain
by
utilizing \eqref{APrioriLqEstPP} and the properties of the trace operator
\begin{align}\label{TraceGaldi}
\TDN\colon\LR{q}\bp{\torus; \WSR{1}{q}\np\Omega}\to\LR{q}\bp{\torus; \WSR{1 - \frac{1}{q}}{q}\np{\torusnn}}, \qquad \phi\mapsto \phi_{|x_3=0},
\end{align}
stated in \cite[Theorem II.4.3]{Galdi} that
\begin{align*}
\norm\F_{\LR{q}\np{\torus; \WSR{1-\frac{1}{q}}{q}\np\torusnn}} &\leq \Cc{c}\bp{\norm\h_{\LR{q}\np{\torus; \WSR{1-\frac{1}{q}}{q}\np\torusnn}} + \norm{\grad\wvel}_{\LR{q}\np{\torus; \WSR{1}{q}\np\Omegap}} + \norm{\tilde\wpres}_{\LR{q}\np{\torus; \WSR{1}{q}\np{\Omega}}}} \\
&\leq \Cc{c}\bp{\norm\h_{\LR{q}\np{\torus; \WSR{1-\frac{1}{q}}{q}\np\torusnn}} + \norm\wvel_{\WSR{1,2}{q}\np\Omegap} + \norm{\grad\wpres}_{\LR{q}\np{\torus\times\Omegap}}},
\end{align*}
which together with \eqref{APrioriLqEstPP} completes the proof.
\end{proof}

\begin{rem}\label{DampingEffectInFourierSpace}
As a key observation in the proof above, we recognize, quantified in the formula \eqref{POP_FS_SolFormulaMitEta}, the damping effect both the viscous fluid and the internal damping has on the displacement $\eta$ of the plate. 
Indeed, without the coupling of the viscous fluid and the introduction of internal damping the representation formula would read 
\begin{align}\label{EtaRepFormulaWithoutdamping}
\eta = \iFT_{\torus\times\torusnn}\BBb{
\frac{\bp{1 - \delta_{\frac{2\pi}{\per}\Z}(k)} \bp{1-\delta_{\np{\frac{2\pi}{\pert}\Z}^2}(\xip)}}{
\snorm{\xip}^4 - k^2  }
\widehat{\F}}.
\end{align}
The additional terms 
\begin{align}\label{DampingEffectInFourierSpaceDampingTerm}
 - \frac{k^2}{\snorm{\xip}} + ik\Bp{\snorm{\xip} + \sqrt{\snorm{\xip}^2 + ik}} \quad\tand \quad ik\snorm{\xip}^2
\end{align}
in the denominator in \eqref{POP_FS_SolFormulaMitEta} manifest the damping effect of the viscous fluid and internal damping, respectively.
Observe that the multiplier in \eqref{EtaRepFormulaWithoutdamping} is unbounded, whereas the damping terms in \eqref{POP_FS_SolFormulaMitEta} lead to a bounded multiplier that decays to 0 as $\snorm{\np{k,\xip}}\ra\infty$.
The decay rate of the multiplier in \eqref{POP_FS_SolFormulaMitEta} as $\snorm{\np{k,\xip}}\ra\infty$ dictates the order of the a priori estimates that can be established
for $\eta$ in terms of the data, and can thus be interpreted as a quantification of the damping effects. 
\end{rem}

Next we seek to utilize the $\LR{q}$ estimate established in the half-space case in Lemma \ref{POPFSExistence} to the original problem \eqref{LinearizedSystem} set in the cuboid $\Omega$. To this end, we employ a standard localization argument. Since \eqref{LinearizedSystem} contains a non-homogeneous Stokes problem, the lower order pressure term that appears naturally in such a localization argument poses a non-trivial challenge.

\begin{lem}[Pressure Field Estimates]\label{LemPressureFieldEst}
Let $s\in (1,\infty)$,
\begin{align*}
\np{\ff, \fp}\in\LRcompl{s}\np{\torus\times\Omega}^3 \times \LRcompl{s}\bp{\torus; \WSR{1-\frac{1}{s}}{s}\np\torusnn}
\end{align*}
and $\np{\uvel, \upres, \eta}\in\SSsigmacompl{s}\np{\torus\times\Omega}$
be a solution to \eqref{LinearizedSystem} with $g=0$.
Then there exists a constant $\Cc{C} = \Cclast{C}\np{\Omega, s}>0$ such that
\begin{align}\label{PressureFieldLqEst}
\begin{aligned}
&\norm{\upres\np{t, \cdot}}_{\LR{\frac{3}{2}s}\np{\Omega}} \leq \Cclast{C}\bp{\norm{\grad\uvel\np{t, \cdot}}_{\LR{s}\np{\Omega}} + \norm{\gradp\Deltap\eta\np{t, \cdot}}_{\LR{s}\np{\torusnn}} + \norm{\gradp\pt\eta\np{t, \cdot}}_{\LR{s}\np{\torusnn}} \\
&\qquad\quad + \norm{\f\np{t, \cdot}}_{\LR{s}\np\Omega} + \norm{\fp\np{t, \cdot}}_{\WSR{1-\frac{1}{s}}{s}\np{\torusnn}} + \norm{\grad\uvel\np{t, \cdot}}_{\LR{s}\np{\Omega}}^{\frac{s-1}{s}} \norm{\grad\uvel\np{t, \cdot}}_{\WSR{1}{s}\np{\Omega}}^{\frac{1}{s}}}
\end{aligned}
\end{align}
for a.e. $t\in\torus$.
Moreover, there is a constant $\Cc{C} = \Cclast{C}\np{\Omega, s}>0$ such that for a.e.
$t\in\torus$ the additional estimate
\begin{align}\label{GradPressureFieldLqEst}
\norm{\grad\upres\np{t, \cdot}}_{\LR{s}\np{\torusnn\times \np{\frac{1}{3}, \frac{2}{3}}}} \leq \Cclast{C}\bp{\norm{\f\np{t, \cdot}}_{\LR{s}\np\Omega} + \norm{\upres\np{t, \cdot}}_{\LR{s}\np{\Omega}}}
\end{align}
holds.
\end{lem}

\begin{proof}
We use the approach from \cite[Proof of Lemma 5.4]{GaldiKyed}.
For simplicity, we do not explicitly denote the $t$-dependency of functions in the following.

First consider an arbitrary $\phi\in\CRi\np{\overline{\Omega}}$, and observe that
due to the periodicity of $\uvel$ and $\phi$, as well as the boundary condition
$\uvel_{|x_3=0} = -\pt\eta e_3$, 
\begin{align}\label{TimeDerivInt}
\int_{\Omega} \pt\uvel \cdot \grad\phi \dx = \int_{\torusnn}\pt^2\eta \, \phi \dxp
\end{align}
holds for a.e. $t\in\torus$.
Hence, by multiplication of \eqrefsub{LinearizedSystem}{2} with
$\grad\phi$ we identify $\upres$ as a
solution to the weak Laplace problem with homogeneous Robin 
and Neumann boundary conditions on the bottom and top face of $\Omega$, respectively,
\ie, for any $\phi\in\CRi\np{\overline{\Omega}}$ 
\begin{align}\label{RobinWeakFormulation}
\begin{aligned}
\int_\Omega \grad\upres \cdot \grad\phi \dx + \int_{\torusnn} \upres \, \phi \dxp &= \int_\Omega \f \cdot \grad\phi \dx + \int_\Omega \Delta\uvel \cdot \grad\phi \dx \\
&\quad - \int_{\torusnn} \fp\phi \dxp + \int_{\torusnn} \Bilaplace\eta \phi \dxp - \int_{\torusnn} \Deltap\pt\eta \phi \dxp.
\end{aligned}
\end{align}
For $\g\in\CRci\np{\Omega}$, existence of a solution $\Phi$ to
\begin{align}
\begin{pdeq}
-\Delta\Phi &= \g && \tin\Omega, \\
\partial_\nu\Phi_{|x_3=0} + \Phi_{|x_3=0} &= 0 && \ton\torusnn, \\
\partial_\nu\Phi_{|x_3=1} &= 0 && \ton\torusnn
\end{pdeq}
\end{align}
obeying for any $q\in (1,\infty)$ the
$\LR{q}$ estimate
\begin{align}\label{RobinEst}
\norm{\Phi}_{\WSR{2}{q}\np\Omega} \leq \Cc{c}\norm{\g}_{\LR{q}\np\Omega}
\end{align}
follows by classical methods; one may mimic the proof for the pure Neumann problem in \cite{SLq90}.
 Then $\Phi$ obeys the weak formulation
\begin{align*}
\int_\Omega \grad\Phi\cdot\grad\psi \dx + \int_{\torusnn} \Phi\psi \dxp = \int_\Omega \g\psi \dx
\end{align*}
for any $\psi\in\CRi\np{\overline\Omega}$.
In view of \eqref{RobinWeakFormulation}, we can thus compute
\begin{align*}
\int_{\Omega} \upres \g \dx &= -\int_\Omega \upres \Delta\Phi \dx = \int_\Omega \grad\upres \cdot \grad\Phi \dx + \int_{\torusnn} \upres \, \Phi \dxp = I_1 + \ldots + I_5,  
\end{align*}
with 
\begin{align*}
\begin{array}{lll}
I_1 \coloneqq \int_\Omega \f \cdot \grad\Phi \dx, & I_2\coloneqq \int_\Omega \Delta\uvel \cdot \grad\Phi \dx, & I_3\coloneqq -\int_{\torusnn} \fp\Phi \dxp \\
I_4\coloneqq \int_{\torusnn} \Bilaplace\eta \, \Phi \dxp, & I_5\coloneqq - \int_{\torusnn} \Deltap\pt\eta \, \Phi \dxp. & 
\end{array}
\end{align*}
By Sobolev embedding we find that 
\begin{align*}
\norm{\Phi}_{\WSR{1}{s'}\np\Omega} \leq \Cc{c}\norm{\Phi}_{\WSR{2}{\frac{3s}{3s-2}}}\leq \Cc{c}\norm\g_{\LR{\frac{3s}{3s-2}}\np\Omega}.
\end{align*}
Using  H\"older's inequality, we deduce
\begin{align*}
\snorm{I_1} \leq \norm\f_{\LR{s}\np\Omega} \norm{\grad\Phi}_{\LR{s'}\np\Omega} \leq \Cc c \norm\f_{\LR{s}\np\Omega} \norm{\g}_{\LR{\frac{3s}{3s-2}}\np\Omega}.
\end{align*}
To find a similar estimate for $I_2$, we twice integrate by parts to deduce
\begin{align*}
I_2 &= \int_{\partial\Omega} \partial_{x_j}\uvel_i \, \partial_{x_i}\Phi \, \normalvec_j \dxp - \int_\Omega \partial_{x_j}\uvel_i \, \partial_{x_j}\partial_{x_i}\Phi \dx \\
&= \int_{\partial\Omega} \np{\partial_{x_j}\uvel_i \, \partial_{x_i}\Phi \, \normalvec_j - \partial_{x_j}\uvel_i \, \partial_{x_j}\Phi \, \normalvec_i} \dxp,
\end{align*}
where we utilized the Einstein summation convention and the fact that $\Div\uvel= 0$.
Hence, applying H\"older's inequality,
Sobolev embeddings and a trace inequality (see
\cite[Theorem II.4.1]{Galdi}), we obtain
\begin{align*}
\snorm{I_2} 
&\leq \Cc c \norm{\grad\uvel}_{\LR{s}\np{\partial\Omega}} \norm{\grad\Phi}_{\LR{s'}\np{\partial\Omega}} \leq \Cc{c} \norm{\grad\uvel}_{\LR{s}\np{\partial\Omega}} \norm{\grad\Phi}_{\WSR{1}{\frac{3s}{3s-2}}\np{\Omega}} \\
&\leq \Cc{c} \bbp{\norm{\grad\uvel}_{\LR{s}\np{\Omega}} + \norm{\grad\uvel}_{\LR{s}\np{\Omega}}^{\frac{s-1}{s}} \norm{\grad\uvel}_{\WSR{1}{s}\np{\Omega}}^{\frac{1}{s}}} \norm{\g}_{\LR{\frac{3s}{3s-2}}\np{\Omega}}.
\end{align*}
The estimates for the final three integrals will be established similarly to
the estimates of $I_2$ by an application of the same trace inequality as above.
It follows that
\begin{align*}
\snorm{I_3} \leq \norm{\fp}_{\LR{s}\np{\torusnn}} \norm{\Phi}_{\LR{s'}\np{\torusnn}} \leq \Cc{c} \norm{\fp}_{\LR{s}\np{\torusnn}} \norm{\Phi}_{\WSR{1}{s'}\np{\Omega}} \leq \Cc{c} \norm{\fp}_{\LR{s}\np{\torusnn}} \norm{\g}_{\LR{\frac{3s}{3s-2}}\np\Omega}.
\end{align*}
In order to utilize the same arguments as for $I_2$, we integrate by
parts in $I_4$ and $I_5$ to find
\begin{align*}
&\snorm{I_4} = \snormL{\int_{\torusnn} \gradp\Deltap\eta \cdot \gradp\Phi \dxp} \leq \Cc{c} \norm{\gradp\Deltap\eta}_{\LR{s}\np{\torusnn}} \norm{\g}_{\LR{\frac{3s}{3s-2}}\np{\Omega}}, \\
&\snorm{I_5} = \snormL{\int_{\torusnn} \gradp\pt\eta \cdot \gradp\Phi \dxp} \leq \Cc{c} \norm{\gradp\pt\eta}_{\LR{s}\np{\torusnn}} \norm{\g}_{\LR{\frac{3s}{3s-2}}\np{\Omega}}.
\end{align*}
It follows from the estimates of $I_1$--$I_5$ that
\begin{align*}
\snormL{\int_{\Omega} \upres \g \dx} &\leq \Cc{c}\bbp{\norm{\f}_{\LR{s}\np\Omega} + \norm{\fp}_{\LR{s}\np{\torusnn}} + \norm{\grad\uvel}_{\LR{s}\np{\Omega}}^{\frac{s-1}{s}} \norm{\grad\uvel}_{\WSR{1}{s}\np{\Omega}}^{\frac{1}{s}} \\
&\quad + \norm{\grad\uvel}_{\LR{s}\np{\Omega}} + \norm{\gradp\Deltap\eta}_{\LR{s}\np{\torusnn}} + \norm{\gradp\pt\eta}_{\LR{s}\np{\torusnn}}} \norm{\g}_{\LR{\frac{3s}{3s-2}}\np{\Omega}}.
\end{align*}
By using the duality $\bp{\LR{\frac{3}{2}s}\np\Omega}' = \LR{\frac{3s}{3s-2}}\np{\Omega}$, we obtain \eqref{PressureFieldLqEst}.

To show \eqref{GradPressureFieldLqEst}, we
let $\chi\in\CRci\np\R$ be a cut-off function such that
\begin{align*}
\chi\np{x_3} = 1 \quad\text{for } x_3\in \bbp{\frac{1}{3}, \frac{2}{3}}, \quad\text{and }\quad \chi\np{x_3} = 0 \quad\text{for } x_3\in \R\setminus\bbb{\frac{1}{6}, \frac{5}{6}}.
\end{align*}
We then set $\wpres\np{x}\coloneqq\chi\np{x_3}\upres\np{x}$,
$\f_1\coloneqq 2\upres\grad\chi + \f\chi$ and
$\f_2\coloneqq \upres\Delta\chi + \f\cdot\grad\chi$,
and observe that $\wpres$
is a solution to the weak Neumann problem
\begin{align}\label{WeakNeumannPressure}
\int_\Omega \grad\wpres\cdot\grad\phi \dx = \int_\Omega \f_1 \cdot \grad\phi \dx + \int_\Omega \f_2 \, \phi \dx \qquad \forall\phi\in\CRci\np{\overline\Omega}.
\end{align}
Since $\supp\f_1, \, \supp\f_2 \subset \torusnn\times \bp{\frac{1}{6}, \frac{5}{6}}$,
we clearly have
\begin{align*}
&\sup_{\norm{\grad\phi}_{\LR{s'}\np\Omega} = 1}\snormL{\int_\Omega \f_1 \cdot \grad\phi \dx} \leq \Cc{c}\bp{\norm{\f}_{\LR{s}\np\Omega} + \norm{\upres}_{\LR{s}\np{\Omega}}}, \\
&\sup_{\norm{\grad\phi}_{\LR{s'}\np\Omega} = 1}\snormL{\int_\Omega \f_2 \cdot \phi \dx} \leq \Cc{c}\bp{\norm{\f}_{\LR{s}\np\Omega} + \norm{\upres}_{\LR{s}\np{\Omega'}}}.
\end{align*}
Finally, \eqref{GradPressureFieldLqEst} follows via a standard \textit{a priori} estimate
for the weak Neumann problem \eqref{WeakNeumannPressure} and $\upres_{|\torusnn\times\np{\frac{1}{3}, \frac{2}{3}}} = \wpres_{|\torusnn\times\np{\frac{1}{3}, \frac{2}{3}}}$.
\end{proof}

We can now establish $\LR{q}$ estimates for the purely periodic part of problem \eqref{LinearizedSystem}.

\begin{lem}\label{LemLqEstPeriLayer}
Let $q\in (1, \infty)$. For every $\np{\f, \fp}\in\LRcompl{q}\np{\torus\times\Omega}^3\times\LRcompl{q}\bp{\torus; \WSR{1-\frac{1}{q}}{q}\np\torusnn}$ and $\g=0$ there is a unique solution 
\begin{align*}
\np{\uvel, \upres, \eta}\in\SSsigmacompl{q}\np{\torus\times\Omega}
\end{align*}
to \eqref{LinearizedSystem}. Moreover
\begin{align}\label{EstFSPeriLayer}
\begin{aligned}
&\norm{\uvel}_{\WSR{1,2}{q}\np{\torus\times\Omega}} + \norm{\grad\upres}_{\LR{q}\np{\torus\times\Omega}} + \norm{\eta}_{\WSR{2}{q}\np{\torus; \WSR{1-\frac{1}{q}}{q}\np{\torusnn}}\cap \LR{q}\np{\torus; \WSR{5-\frac{1}{q}}{q}\np{\torusnn}}} \\
&\qquad\qquad\qquad\qquad\qquad\qquad\qquad \leq \Cclast{C}\bp{\norm\f_{\LR{q}\np{\torus\times\Omega}} + \norm\fp_{\LR{q}\np{\torus; \WSR{1-\frac{1}{q}}{q}\np\torusnn}}}
\end{aligned}
\end{align}
with $\Cc{C} = \Cclast{C}\np{q, \per, \pert}>0$.
\end{lem}

\begin{proof}
It suffices to show the statement for $\np{\ff,\fp}\in\CRcicompl\np{\torus\times\Omega}^3\times\CRcicompl\np{\torust}$; the general case then follows via \eqref{EstFSPeriLayer} by a density argument.
Existence and uniqueness follow from Lemma \ref{EsistenceLinTPProblem} and Lemma \ref{FSLinUniquenessLem}.
To show \eqref{EstFSPeriLayer}, we
mimic the steps in \cite[Proof of Theorem 5.1]{GaldiKyed}
and employ Lemma \ref{POPFSExistence}.
Let
$\chi\in\CRi(\R)$ such that
\begin{align*}
\chi\np{x_3} = 1 \quad\text{for } x_3\leq \frac{1}{3}, \quad\text{and }\quad \chi\np{x_3} = 0 \quad\text{for } x_3\geq \frac{2}{3}.
\end{align*}
We define
the even and odd extensions of a function $\phi$ as
\begin{align*}
\phi^e(t, x) &:= \begin{cases} \phi(t, x',x_{3}) &  0\leq x_{3}\leq 1, \\ (1 - \chi(2-x_{3}))\phi(t, x', 2-x_{3}) &  1<x_{3}, \end{cases} \\
\phi^o(t, x) &:= \begin{cases} \phi(t, x',x_{3}) &  0\leq x_{3}\leq 1, \\ -(1 - \chi(2-x_{3}))\phi(t, x', 2-x_{3}) &  1<x_{3}. \end{cases}
\end{align*}
Letting
\begin{align*}
\Uvel\coloneqq\np{\uvel_1^e, \uvel_2^e, \uvel_3^o} \qquad\text{and}\qquad       \Upres\coloneqq \upres^e,
\end{align*}
we have extended $\np{\uvel, \upres, \eta}$ to a solution
$\np{\Uvel, \Upres, \eta}$ to \eqref{FullyInhomLinearizedSystemtem}
in the half space with right-hand side
$\np{\f, \g, \h} = \np{\fr, \gr, \fp}$ with 
\begin{align}\label{LqEstExtF}
\begin{aligned}
&\fr_j\np{t, x} := \begin{cases}
\f_j\np{t, x} & \text{if } x_3<1, \\
\varepsilon_j \bp{(1 - \chi(2-x_3))\f_j(t, x', 2-x_3) - H_j\np{t, x}} & \text{if } x_3>1,
\end{cases} \\
&\gr\np{t, \xprime, x_3} := \begin{cases}
0 & \text{if } x_3<1, \\
-\chi'\np{2-x_3}\,\uvel_3\np{t, \xprime, 2-x_3} & \text{if } x_3>1,
\end{cases}
\end{aligned}
\end{align}
where 
$\varepsilon_j := 1$ for $j=1,2$, $\varepsilon_3 := -1$, and
\begin{align}\label{Defh}
\begin{aligned}
H\np{t, x} := &\chi''\np{2-x_3}\uvel\np{t,\xprime,2-x_3} + 2\partial_{x_3}\uvel\np{t,\xprime,2-x_3}\chi'\np{2-x_3} \\
&+ \chi'\np{2-x_3} \, \upres\np{t,\xprime, 2-x_3} \, e_3.
\end{aligned}
\end{align}
Observer that $\Uvel\np{t,\xprime,0} = \uvel\np{t,\xprime,0}$ and $\Upres\np{t,\xprime,0} = \upres\np{t,\xprime,0}$.
Due to the identity
\begin{align*}
\chi'\Delta\uvel_3 = \Div\nb{\chi' \grad\uvel_3} - \partial_{x_3}\uvel_3 \chi''
\end{align*}
we deduce by \cite[Theorem III.3.4]{Galdi} and \eqref{BogovskiiEst1} that
\begin{align}\label{EstLqBogogvskii}
\begin{aligned}
\norm{\Bogovskii\np{\chi'\Delta\uvel_3}}_{\LR{q}\np{\torus\times\Omega}} &\leq \norm{\Bogovskii\np{\Div\nb{\chi' \grad\uvel_3}}}_{\LR{q}\np{\torus\times\Omega}} + \norm{\Bogovskii\np{\partial_{x_3}\uvel_3 \chi''}}_{\LR{q}\np{\torus\times\Omega}} \\
&\leq \Cc{c}\np{\chi}\norm{\uvel}_{\LR{q}\np{\torus; \WSR{1}{q}\np\Omega}}.
\end{aligned}
\end{align}
Letting $\Vvel\coloneqq\Bogovskii\np{\chi' \uvel_3}$, we have $\supp\Vvel\subset\QT\coloneqq\torust\times\bp{\frac{1}{3}, \frac{2}{3}}$ and
in view of \eqrefsub{FullyInhomLinearizedSystemtem}{2} and
the identity $\pt\Vvel=\Bogovskii\np{\chi' \pt\uvel_3}$ that
\begin{align*}
\norm{\pt\Vvel}_{\LR{q}\np{\torus\times\Omega}} &\leq \Cc{c}\bp{\norm{\Bogovskii\np{\chi' \f_3}}_{\LR{q}\np{\QT}} + \norm{\Bogovskii\np{\chi' \Delta\uvel_3}}_{\LR{q}\np{\QT}} + \norm{\Bogovskii\np{\chi' \partial_{x_3}\upres}}_{\LR{q}\np{\QT}}} \\
&\leq \Cc{c}\np\chi\bp{\norm{\f}_{\LR{q}\np{\torus\times\Omega}} + \norm{\uvel}_{\LR{q}\np{\torus; \WSR{1}{q}\np\Omega}} + \norm{\grad\upres}_{\LR{q}\np{\QT}}}.
\end{align*}
Now utilizing \eqref{GradPressureFieldLqEst} and \eqref{BogovskiiEst1},
we obtain 
\begin{equation}\label{LqEstFSDiv}
\norm\Vvel_{\WSR{1,2}{q}\np{\torus\times\Omegap}} \leq \Cc{c}\bp{\norm\f_{\LR{q}\np{\torus\times\Omega}} + \norm{\uvel}_{\LR{q}\np{\torus; \WSR{1}{q}\np\Omega}} + \norm{\upres}_{\LR{q}\np{\torus\times\Omega}}}.
\end{equation}
By setting
\begin{align*}
\wvel\colon\torus\times\Omegap\to\R^3, \quad \wvel = \Uvel - \Vvel, \qquad\text{and}\qquad \wpres\colon\torus\times\Omegap\to\R, \quad \wpres = \Upres,
\end{align*}
we obtain a solution $\np{\wvel, \wpres, \eta}\in\SSsigmacompl{q}\np{\torus\times\Omegap}$
to the half space problem \eqref{FullyInhomLinearizedSystemtem} with
$\np{\f, \g, \h} = \np{\fr-\nb{\pt - \Delta}\Vvel, 0, \fp}$.
Moreover, due to Lemma \ref{POPFSExistence} and
\eqref{LqEstFSDiv} $\np{\wvel, \wpres, \eta}$ obeys
\begin{align}\label{LqEstFSwvel}
\begin{aligned}
&\norm{\wvel}_{\WSR{1,2}{q}\np{\torus\times\Omegap}} + \norm{\grad\wpres}_{\LR{q}\np{\torus\times\Omegap}} + \norm{\eta}_{\WSR{2}{q}\np{\torus; \WSR{1-\frac{1}{q}}{q}\np{\torusnn}}\cap \LR{q}\np{\torus; \WSR{5-\frac{1}{q}}{q}\np{\torusnn}}}\\
& \quad \leq \Cc{C}\bp{\norm\fr_{\LR{q}\np{\torus\times\Omegap}} + \norm\f_{\LR{q}\np{\torus\times\Omega}} + \norm\fp_{\LR{q}\np{\torus; \WSR{1-\frac{1}{q}}{q}\np\torusnn}} \\
& \qquad \qquad \quad + \norm{\uvel}_{\LR{q}\np{\torus; \WSR{1}{q}\np\Omega}} + \norm{\upres}_{\LR{q}\np{\torus\times\Omega}}}.
\end{aligned}
\end{align}
In view of \eqref{LqEstExtF} and \eqref{Defh},
we deduce
\begin{align*}
\norm\fr_{\LR{q}\np{\torus\times\Omegap}} &\leq \bp{\norm\f_{\LR{q}\np{\torus\times\Omega}} + \norm{\frr}_{\LR{q}\np{\torust\times (1, \infty)}} + \norm{\h}_{\LR{q}\np{\torust\times (1, \infty)}}} \\
&\leq \Cc{c}\np{\chi}\bp{\norm\f_{\LR{q}\np{\torus\times\Omega}} + \norm{\uvel}_{\LR{q}\np{\torus; \WSR{1}{q}\np\Omega}} + \norm{\upres}_{\LR{q}\np{\torus\times\Omega}}}
\end{align*}
with
\begin{align*}
\frr\np{t,x} := (1 - \chi(2-x_3))\f(t, x', 2-x_3),
\end{align*}
where the second inequality above follows by
utilizing H\"older's inequality and shifting the coordinates in 
the second and third norm above.
Therefore, we conclude from $\np{\uvel, \upres, \eta} = \np{\wvel_{|\torus\times\Omega}, \wpres_{|\torus\times\Omega}, \eta} + \np{\Vvel_{|\torus\times\Omega}, 0, 0}$
as well as the estimates \eqref{LqEstFSDiv} and \eqref{LqEstFSwvel} that
\begin{align}\label{BootStrapEst1}
\begin{aligned}
&\norm{\uvel}_{\WSR{1,2}{q}\np{\torus\times\Omega}} + \norm{\grad\upres}_{\LR{q}\np{\torus\times\Omega}} + \norm{\eta}_{\WSR{2}{q}\np{\torus; \WSR{1-\frac{1}{q}}{q}\np{\torusnn}}\cap \LR{q}\np{\torus; \WSR{5-\frac{1}{q}}{q}\np{\torusnn}}} \\
&\qquad \leq \Cc{c}\bp{\norm\f_{\LR{q}\np{\torus\times\Omega}} + \norm\fp_{\LR{q}\np{\torus; \WSR{1-\frac{1}{q}}{q}\np\torusnn}} + \norm{\uvel}_{\LR{q}\np{\torus; \WSR{1}{q}\np\Omega}} + \norm{\upres}_{\LR{q}\np{\torus\times\Omega}}}.
\end{aligned}
\end{align}
In order to complete the proof, it remains to show that the final two terms on the right-hand side in \eqref{BootStrapEst1} can be omitted. For this purpose,
we utilize Young's inequality and Ehrling's Lemma to deduce
\begin{align*}
&\int_\torus \bp{\norm{\grad\uvel}_{\LR{q}\np{\Omega}}^{\frac{q-1}{q}} \norm{\grad\uvel\np{t, \cdot}}_{\WSR{1}{q}\np{\Omega}}^{\frac{1}{q}}}^q \dt \leq \Cc{c}\np q \varepsilon^{-\frac{1}{q-1}}\norm{\grad\uvel}_{\LR{q}\np{\torus\times\Omega}}^{q} + \varepsilon\norm{\grad\uvel}_{\LR{q}\np{\torus; \WSR{1}{q}\np\Omega}}^{q},
\end{align*}
for any $\varepsilon>0$, as well as
\begin{align*}
\norm\uvel_{\LR{q}\np{\torus; \WSR{1}{q}\np\Omega}} \leq \Cc{c}\np{\delta}\norm\uvel_{\LR{q}\np{\torus\times\Omega}} + \delta\norm\uvel_{\LR{q}\np{\torus; \WSR{2}{q}\np\Omega}}
\end{align*}
with $\delta>0$.
Consequently,
\eqref{PressureFieldLqEst} and \eqref{BootStrapEst1} yield
\begin{align*}
\begin{aligned}
&\norm\uvel_{\WSR{1,2}{q}\np{\torus\times\Omega}} + \norm{\grad\upres}_{\LR{q}\np{\torus\times\Omega}} + \norm{\eta}_{\WSR{2}{q}\np{\torus; \WSR{1-\frac{1}{q}}{q}\np{\torusnn}}\cap \LR{q}\np{\torus; \WSR{5-\frac{1}{q}}{q}\np{\torusnn}}}\\
& \qquad \leq \Cc[Ehrling3]{c}\bp{\norm\f_{\LR{q}\np{\torus\times\Omega}} + \norm\fp_{\LR{q}\np{\torust}}  + \norm\uvel_{\LR{q}\np{\torus\times\Omega}} \\
& \qquad \qquad + \norm{\gradp\Deltap\eta}_{\LR{q}\np{\torus\times\torusnn}} + \norm{\gradp\pt\eta}_{\LR{q}\np{\torus\times\torusnn}}} + \Cc[Ehrling4]{c}\norm{\uvel}_{\LR{q}\np{\torus; \WSR{2}{q}\np\Omega}},
\end{aligned}
\end{align*}
with $\delta = \delta\np\varepsilon>0$, $\const{Ehrling3} = \const{Ehrling3}(q, \delta, \varepsilon)>0$ and $\const{Ehrling4} = \const{Ehrling4}(q, \delta, \varepsilon)>0$.
Observe that we have utilized Ehrling's Lemma to estimate the
third norm on the right-hand side in \eqref{BootStrapEst1}.
Choosing $\varepsilon$ and $\delta$ sufficiently small, we finally deduce
\begin{align}\label{LqEstWithLowerOrder}
\begin{aligned}
&\norm\uvel_{\WSR{1,2}{q}\np{\torus\times\Omega}} + \norm{\grad\upres}_{\LR{q}\np{\torus\times\Omega}} + \norm{\eta}_{\WSR{2}{q}\np{\torus; \WSR{1-\frac{1}{q}}{q}\np{\torusnn}}\cap \LR{q}\np{\torus; \WSR{5-\frac{1}{q}}{q}\np{\torusnn}}}\\
& \quad \leq \Cc{c}\bp{\norm\f_{\LR{q}\np{\torus\times\Omega}} + \norm\fp_{\LR{q}\np{\torust}} + \norm\uvel_{\LR{q}\np{\torus\times\Omega}}\\
& \qquad \qquad + \norm{\gradp\Deltap\eta}_{\LR{q}\np{\torus\times\torusnn}} + \norm{\gradp\pt\eta}_{\LR{q}\np{\torus\times\torusnn}}}.
\end{aligned}
\end{align}
Since the embeddings
\begin{align*}
&\WSR{1,2}{q}\np{\torus\times\Omega}\hookrightarrow\LR{q}\np{\torus\times\Omega}, \\
&\WSR{2}{q}\bp{\torus; \WSR{1-\frac{1}{q}}{q}\np{\torusnn}}\cap \LR{q}\bp{\torus; \WSR{5-\frac{1}{q}}{q}\np{\torusnn}}\hookrightarrow\WSR{1}{q}\bp{\torus; \WSR{1}{q}\np\Omega}, \\
&\WSR{2}{q}\bp{\torus; \WSR{1-\frac{1}{q}}{q}\np{\torusnn}}\cap \LR{q}\bp{\torus; \WSR{5-\frac{1}{q}}{q}\np{\torusnn}}\hookrightarrow\LR{q}\bp{\torus; \WSR{3}{q}\np\Omega}, 
\end{align*}
are compact and the solution to \eqref{LinearizedSystem}
with homogeneous right-hand side is zero by Lemma \ref{FSLinUniquenessLem},
the $\LR{q}$ estimate \eqref{EstFSPeriLayer}
follows by a standard contradiction argument.
\end{proof}

\subsection{Steady-state problem}\label{SteadyStateSection}

Applying the projection $\PR$ to the linear system \eqref{LinearizedSystem}, we obtain 
\begin{align}\label{FSLqEstSS}
\begin{pdeq}
\Bilaplace\etas &= \fps -e_3\cdot\bp{\fluidstress(\uvels,\upress)e_3}_{| x_3= 0} && \tin\torusn^2, \\
-\Delta\uvels + \grad\upress &= \ffs && \tin\Omega, \\
\Div\uvels &= \g_s && \tin\Omega, \\
\uvels &= 0 && \ton \partial\Omega,\\
\int_{\torusnn} \etas\np{\xprime} \dxp &= 0. &&
\end{pdeq}
\end{align}
This weakly coupled system of elliptic equations (Stokes problem coupled with the bi-harmonic equation) is relatively simple to analyze.

\begin{lem}\label{PropLqEstPeriLayerSS}
Let $q\in (1, \infty)$. For $\np{\ffs, \fps}\in\LR{q}\np\Omega^3\times\WSR{1-\frac{1}{q}}{q}\np\torusnn$ and $\g_s=0$
there is a unique solution
$\np{\uvels, \upress, \etas}\in\WSR{2}{q}\np\Omega \times \WSR{1}{q}\np\Omega \times \WSR{5-\frac{1}{q}}{q}\np\torusnn$
to \eqref{FSLqEstSS} satisfying
\begin{align}\label{LqEstSSFS}
\norm\uvels_{\WSR{2}{q}\np\Omega} + \norm{\upress}_{\WSR{1}{q}\np\Omega} + \norm\etas_{\WSR{5-\frac{1}{q}}{q}\np\torusnn} \leq \Cc{C}\bp{\norm\ffs_{\LR{q}\np\Omega} + \norm{\fps}_{\WSR{1-\frac{1}{q}}{q}\np\torusnn}}
\end{align}
with $\Cclast{C} = \Cclast{C}\np{q} > 0$.
\end{lem}

\begin{proof}
Existence of a solution $(\uvels,\tupress)$ to the steady-state Stokes system \eqrefsub{FSLqEstSS}{2-4} that is unique up-to addition of a constant to $\tupress$ and satisfying 
\begin{align}\label{EstSSStokes}
\norm\uvels_{\WSR{2}{q}\np\Omega} + \norm{\grad\tupress}_{\LR{q}\np\Omega} \leq \Cc{c}\norm\ffs_{\LR{q}\np\Omega}
\end{align}
follows by the same arguments typically employed in the analysis of the steady-state Stokes problem set in classical bounded subdomains of $\R^n$; see for example \cite[Theorem IV.6.1]{Galdi}.
Letting
\begin{align}\label{PropLqEstPeriLayerSSNormalization}
\upress\coloneqq \tupres - \frac{1}{\snorm\torusnn}\int_{\torusnn} \fps \dxp - \frac{1}{\snorm\torusnn}\int_{\torusnn} \tupres \dx
\end{align}
we obtain a solution to 
\begin{align}\label{FSLqEstBiLaplacian}
\begin{pdeq}
&\Bilaplace\etas = \fps  -e_3\cdot\bp{\fluidstress(\uvels,\upress)e_3}_{|x_3 = 0}  \qquad\tin\torusn^2,\\
&\int_{\torusnn} \etas\np{\xprime} \dxp = 0
\end{pdeq}
\end{align}
via the formula
\begin{align}\label{PropLqEstPeriLayerSSRepFormula}
\etas\coloneqq\iFT_{\torusnn}\Bb{\frac{1-\delta_{\np{\frac{2\pi}{\pert}\Z}^2}\np\xip}{\snorm{\xip}^4}\FT_{\torusnn}\bb{\fps  -e_3\cdot\bp{\fluidstress(\uvels,\upress)e_3}_{|x_3 = 0} }},
\end{align}
where $\delta_{\frac{2\pi}{\pert}\Z}$ denotes the Dirac measure on the group $\np{\frac{2\pi}{\pert}\Z}^2$.
The term $1-\delta_{\frac{2\pi}{\pert}\Z}\np\xip$ appears in the numerator above due to \eqref{PropLqEstPeriLayerSSNormalization}, which implies 
\begin{align*}
\FT_{\torusnn}\bb{ \fps  -e_3\cdot\bp{\fluidstress(\uvels,\upress)e_3}_{|x_3 = 0}}(0) =0
\end{align*}
since $e_3\cdot\bp{\fluidstress(\uvels,\upress)e_3}_{|x_3 = 0}=\upres_{\mathrm{s} |x_3 = 0}$ on $\torusnn$.
The term $1-\delta_{\frac{2\pi}{\pert}\Z}\np\xip$ in \eqref{PropLqEstPeriLayerSSRepFormula} implies $\int_{\torusnn} \etas\np{\xprime} \dxp = 0$.
Utilizing \textsc{De Leeuw}'s Transference Principle \cite{Leeuw1965} on the Fourier multiplier in \eqref{PropLqEstPeriLayerSSRepFormula}, we obtain
\begin{align*}
\norm{\etas}_{\WSR{5-\frac{1}{q}}{q}\np\torusnn} &\leq \Cc{c}\norm{\fps  -e_3\cdot\bp{\fluidstress(\uvels,\upress)e_3}_{|x_3 = 0}}_{\WSR{1-\frac{1}{q}}{q}\np\torusnn} \\
&\leq\Cc{c}\bp{\norm\fps_{\WSR{1-\frac{1}{q}}{q}\np\torusnn} + \norm{\ffs}_{\LR{q}\np\Omega}}.
\end{align*}
By an application of Poincar\'e's inequality to $\upress$, we conclude \eqref{LqEstSSFS}.
\end{proof}

\subsection{Fully inhomogeneous linearized problem}\label{ProofOfLinMainThmSection}

Finally, we consider the fully inhomogeneous
system \eqref{LinearizedSystem} and establish Theorem \ref{MainThmFSLin}.

\begin{proof}[Proof of Theorem \ref{MainThmFSLin}]                                         
Recall Theorem \ref{BogovskiiThm} and set $\wvel\coloneqq\Bogovskii\np{\g}$. Since $\g$ has a vanishing mean over $\Omega$, $\wvel$ solves 
\begin{align}\label{DivProblemFSReduction}
\begin{pdeq}
\Div\wvel &= \g && \tin\torus\times\Omega, \\
\wvel &= 0 && \ton\torus\times\partial\Omega.
\end{pdeq}
\end{align}
Moreover,
\begin{align*}
&\norm{\wvel\np{t, \cdot}}_{\WSR{2}{q}\np\Omega} \leq \Cc{c}\norm{\g\np{t, \cdot}}_{\WSR{1}{q}\np\Omega}, \\
&\norm{\wvel\np{t, \cdot}}_{\LR{q}\np\Omega} \leq \Cc{c}\snorm{\g\np{t, \cdot}}_{-1, q}^*.
\end{align*}
Introducing the vector-field 
$\vvel \coloneqq \uvel - \wvel$, we reduce the investigation of \eqref{LinearizedSystem} with fully inhomogeneous right-hand side
to that of \eqref{LinearizedSystem} with right-hand side
\begin{align*}
\tff:=\ff - \pt\wvel + \Delta\wvel, \quad
\tg:=0, \quad
\tfp := \fp - \g_{|x_3=0}.
\end{align*}
Specifically, letting $\np{\uvels, \upress, \etas}\in\WSR{2}{q}\np\Omega \times \WSR{1}{q}\np\Omega \times \WSR{5-\frac{1}{q}}{q}\np\torusnn$ be the solution from Lemma \ref{PropLqEstPeriLayerSS} corresponding
to data $\np{\PR\tff,0,\PR\tfp}$, 
and $\np{\uveltp, \uprestp, \etatp}\in\SSsigmacompl{q}\np{\torus\times\Omega}$ the solution from Lemma \ref{LemLqEstPeriLayer} corresponding to data $\np{\oPR\tff,0,\oPR\tfp}$, then
\begin{align*}
\np{\uvel,\upres,\eta} := \np{\uvels+\uveltp+\wvel, \upress+\uprestp, \etas + \etatp} 
\end{align*}
solves \eqref{LinearizedSystem}.
From \eqref{EstFSPeriLayer} and \eqref{LqEstSSFS} it follows that 
\begin{align*}
&\norm{\np{\uvel, \upres, \eta}}_{\SLS{q}\np{\torus\times\Omega}} \leq \Cc{c}\,\norm{\np{\tff, 0, \tfp}}_{\DAS{q}\np{\torus\times\Omega}}\\
& \quad \leq \Cc{c} \bp{ \norm\ff_{\LR{q}\np{\torus\times\Omega}} + \norm\fp_{\LR{q}\np{\torus; \WSR{1-\frac{1}{q}}{q}\np\torusnn}} 
+ \norm{\Bogovskii\np{\pt\g}}_{\LR{q}\np{\torus\times\Omega}} + \norm{\wvel}_{\LR{q}\np{\torus; \WSR{2}{q}\np\Omega}}
}\\
& \quad \leq \Cc{c} \bp{ \norm\ff_{\LR{q}\np{\torus\times\Omega}} + \norm\fp_{\LR{q}\np{\torus; \WSR{1-\frac{1}{q}}{q}\np\torusnn}} 
+ \norm\g_{\LR{q}\np{\torus; \WSR{1}{q}\np\Omega}\cap\WSR{1}{q}\np{\torus; \WSRD{-1}{q}\np\Omega}}
}\\
& \quad\leq \Cc{c}\,\norm{\np{\ff, \g, \fp}}_{\DAS{q}\np{\torus\times\Omega}}.
\end{align*}
\end{proof}

\section{Proof of the main theorem}

Employing Theorem \ref{MainThmFSLin}, we shall now prove Theorem \ref{mainThm} using a fixed point argument. We need the following embedding properties to estimate the nonlinear terms in this approach:

\begin{thm}[Sobolev embedding]\label{SobEmbeddingThm}
Let $\Omega\subset\R^n$ ($n\geq 2$) be the whole space $\R^n$, the half space $\R^n_+$, or a bounded domain with a Lipschitz boundary. Let $q\in(1,\infty)$. Moreover, let $m\in\N$, $\Mt\in\N_0$ and $\mx\in\N_0^n$ such that 
\begin{align*}
0\leq\Mx+2\Mt\leq 2m,
\end{align*}
with $\Mx\coloneqq\snorm{\mx}$, and $\alpha, \beta \in [0, 2\np{m-\Mt}-\Mx]$ such that $\beta\coloneqq 2\np{m-\Mt}-\Mx-\alpha$. Assume that $p, r\in [q, \infty]$ satisfy
\begin{align*}
\begin{pdeq}
&r\leq \frac{2q}{2-\alpha q} && \tif\ \alpha q<2,\\
&r<\infty && \tif\ \alpha q=2,\\
&r\leq\infty && \tif\ \alpha q>2,
\end{pdeq}
\qquad
\begin{pdeq}
&p\leq \frac{nq}{n-\beta q} && \tif\  \beta q<{n},\\
&p<\infty && \tif\ \beta q={n},\\
&p\leq\infty && \tif\ \beta q>{n}.
\end{pdeq}
\end{align*}
Then there exists a constant $\Cc{C} = \Cclast{C}(\per, \Omega, p, q, r)>0$ such that 
\begin{align}\label{SobEmbeddingThmEst}
\norm{\partial_x^{\mx}\pt^{\Mt}\uvel}_{\LR{r}\np{\torus; \LR{p}(\Omega)}}\leq\Cclast{C}\norm{\uvel}_{\WSR{m,2m}{q}(\torus\times\Omega)}
\end{align}
holds for all $\uvel\in\WSRcompl{m,2m}{q}(\torus\times\Omega):=\WSR{m}{q}\bp{\torus; \LR{q}\np\Omega}\cap\LR{q}\bp{\torus; \WSR{2m}{q}\np\Omega}^3$.
\end{thm}

\begin{proof}
See \cite[Theorem 2.6.3]{CelikDiss}.
\end{proof}

To simplify notation, we again set $\mus=\muf=1$.
We start by reformulating \eqref{CoupledSystem} in the reference configuration. To this end, we recall the transformation $\phi_\eta$ from \eqref{Transformation} and
set
\begin{align*}
\tuvel:=\uvel\circ\phi_\eta, \quad \tupres:=\uvel\circ\phi_\eta.
\end{align*}
We thus obtain the following reformulating of \eqref{CoupledSystem} in the reference configuration:
\begin{align}\label{NonLinFSOnReference}
\begin{pdeq}
\pt^2\eta + \Bilaplace\eta- \Deltap\pt\eta &= \fp -e_3\cdot\bp{\fluidstress(\tuvel,\tupres)e_3}_{| x_3= 0}
+ \tRS && \tin\torus\times\torusn^2, \\
\pt\tuvel - \Delta\tuvel + \nsnl{\tuvel} + \grad\tupres &= \tff + \RFtilde && \tin\torus\times\Omega, \\
\Div\tuvel &= \tRD && \tin\torus\times\Omega, \\
\tuvel\np{t, \xprime, 0} &= -\pt\eta\np{t, \xprime} e_3 && \ton\torus\times\torusn^2, \\
\tuvel(t,\xprime, 1) &= 0 && \ton\torus\times\torusnn,\\
\int_{\omega} \eta\np{t, \xprime} \dxp &= 0,
\end{pdeq}
\end{align}
where $\tff:=\ff\circ\phi_\eta$ and the nonlinear terms $\tRS$, $\tRF$ and $\tRD$ are given by
\begin{align*}
&\RFtilde\np{\tuvel, \tupres, \eta} \coloneqq \tRF - \nsnl{\tuvel}, \\
&\tRD\np{\tuvel, \eta} \coloneqq \Div\RD, \\
&\tRS\np{\tuvel, \tupres, \eta} \coloneqq -e_3\cdot\bb{\np{\fluidstress(\tuvel, \tupres)_{| x_3= 0}\np{\tilde\normalvec_t - e_3}} + \Seta\tilde\normalvec_t},
\end{align*}
the normal vectors given by
\begin{align*}
\tilde\normalvec_t \coloneqq \normalvec_t\circ\phi_\eta \qquad\text{and}\qquad \normalvec_t = \np{1 + \snorm{\gradp\eta}^2}^{-\frac{1}{2}}\begin{pmatrix} \gradp\eta \\ -1 \end{pmatrix},
\end{align*}
and the terms $\tRF$, $\RD$ and $\Seta$ given by
\begin{align*}
&\tRF\np{\tuvel, \tupres, \eta} \coloneqq \partial_{x_3}\tuvel\frac{\rho\pt\eta}{1+\eta} + \sum_{j,k,l = 1}^3 \partial_{x_l}\partial_{x_k}\tuvel\np{\EMt{k}{l} + \EMt{l}{k} + \EMt{j}{k}\EMt{j}{l}}\circ\phi_\eta \\
&\qquad\qquad + \sum_{j, k=1}^3 \partial_{x_k}\tuvel\np{\partial_{x_j}\EMt{j}{k}}\circ\phi_\eta - \frac{\partial_{x_3}\tupres}{1+\eta}\begin{pmatrix} \rho\gradp\eta \\ -\eta \end{pmatrix} + \grad\tuvel\np{\Et\circ\phi_\eta}\tuvel, \\
&\RD\np{\tuvel} \coloneqq -\begin{pmatrix} \eta\tuvel' \\ \rho\gradp\eta\cdot\tuvel'  \end{pmatrix}, \\
&\Seta\np{\tuvel, \eta} \coloneqq \bp{\np{\grad\tuvel \np{\Et\circ\phi_\eta} + \np{\grad\tuvel \np{\Et\circ\phi_\eta}}^\top}}_{| x_3= 0},
\end{align*}
where $\rho$ and the matrix $\Eeta$ are defined as
\begin{align*}
\rho\np{x_3}\coloneqq 1-x_3, \qquad\text{and}\qquad
\EMeta{i}{j}\coloneqq\frac{1}{1+\eta}
\begin{cases}
0 & \text{if } i\neq 3, \\
\rho\partial_{x_j}\eta & \text{if } i=3 \text{ and } j\neq 3, \\
-\eta & \text{if } i=3=j.
\end{cases}
\end{align*}
To simplify notation in the following, we set
\begin{align*}
\ESR{q}\np{\torus\times\torusnn}:=\WSR{2}{q}\bp{\torus; \WSR{1-\frac{1}{q}}{q}\np{\torusnn}}\cap \LR{q}\bp{\torus; \WSR{5-\frac{1}{q}}{q}\np{\torusnn}}.
\end{align*}

\begin{lem}\label{EstMatrixTransformation}
Let $q\in (1,\infty)$. Then there exists an $\varepsilonn>0$ and
a constant $\Cc{C} >0$ such that if $\eta$ satisfies
\begin{align}\label{SmallnessEtaNLFS}
\norm\eta_{\ESR{q}\np{\torus\times\torusnn}} \leq \varepsilonn,
\end{align}
then $\Eeta\in\LR{q}\np{\torus\times\Omega}$ and
\begin{align}\label{EstE}
\norm{\Eeta}_{\LR{q}\np{\torus\times\Omega}} \leq \Cclast{C}\norm\eta_{\LR{q}\np{\torus; \WSR{1}{q}\np\torusnn}}.
\end{align}
\end{lem}

\begin{proof}
Due to the definition of $\Eeta$, we have
\begin{align}\label{EstE1}
\norm{\Eeta}_{\LR{q}\np{\torus\times\Omega}} \leq \normL{\frac{\rho\gradp\eta}{1+\eta}}_{\LR{q}\np{\torus\times\Omega}} + \normL{\frac{\eta}{1+\eta}}_{\LR{q}\np{\torus\times\Omega}}.
\end{align}
Utilizing the Sobolev embedding theorem
(Theorem \ref{SobEmbeddingThm}) with $m=2$, $m_x=0=\Mt$
and $\alpha = 2$, we find for any $q>1$ that
\begin{align*}
\norm\eta_{\LR{\infty}\np{\torus\times\torusnn}} \leq \Cc[ConstEmbeddingNonlinFS]{c}\norm\eta_{\WSR{2,4}{q}\np{\torus\times\torusnn}} \leq \Cclast{c}\norm\eta_{\ESR{q}\np\torust} \leq \Cclast{c}\varepsilonn,
\end{align*}
and consequently by choosing $\varepsilonn = \frac{1}{2\Cclast c}$ that
\begin{align*}
\begin{aligned}
\normL{\frac{1}{1+\eta}}_{\LR{\infty}\np{\torus\times\torusnn}} &\leq \frac{1}{1 - \norm\eta_{\LR{\infty}\np{\torus\times\torusnn}}} \leq \frac{1}{1-\const{ConstEmbeddingNonlinFS}\varepsilonn} \leq 2<\infty.
\end{aligned}
\end{align*}
In view of \eqref{EstE1}, this estimate implies \eqref{EstE} since $x_3\in (0,1)$.
\end{proof}

\begin{lem}[Estimates of nonlinear terms]\label{FSNonLinEstNLTerme}
Let $\varepsilonn>0$ and $q\in (2,\infty)$.
Then for any $\np{\uvel, \upres, \eta}\in\SLS{q}\np{\torus\times\Omega}$
satisfying \eqref{SmallnessEtaNLFS}
the nonlinear terms obey 
\begin{align}
&\begin{aligned}
&\norm{\RFtilde}_{\LR{q}\np{\torus\times\Omega}} \leq \Cc{C}\bp{\np{1 + \varepsilonn}\norm{\uvel}_{\WSR{1,2}{q}\np{\torus\times\Omega}} + \norm{\grad\upres}_{\LR{q}\np{\torus\times\Omega}} \\
&\qquad\qquad\qquad\qquad\qquad\qquad\qquad\qquad\qquad + \norm{\uvel}_{\WSR{1,2}{q}\np{\torus\times\Omega}}^2}\norm{\eta}_{\ESR{q}\np{\torus\times\torusnn}},
\end{aligned}\label{RFEst} \\
&\norm{\tRD}_{\LR{q}\np{\torus; \WSR{1}{q}\np\Omega} \cap \WSR{1}{q}\np{\torus; \WSRD{-1}{q}\np\Omega}} \leq \Cc{C}\norm\uvel_{\WSR{1,2}{q}\np{\torus\times\Omega}}\norm\eta_{\ESR{q}\np{\torus\times\torusnn}},\label{RDEst} \\
&\begin{aligned}
&\norm{\tRS}_{\LR{q}\np{\torus;\WSR{1-\frac{1}{q}}{q}\np\torusnn}} \leq \Cc{C}\np{1 + \varepsilonn}\bp{\norm{\eta}_{\ESR{q}\np{\torus\times\torusnn}}\norm\uvel_{\WSR{1,2}{q}\np{\torus\times\Omega}} \\
&\qquad\qquad\qquad\qquad\qquad\qquad\qquad\qquad + \norm{\uvel}_{\WSR{1,2}{q}\np{\torus\times\Omega}} + \norm{\upres}_{\LR{q}\np{\torus; \WSR{1}{q}\np\Omega}}}.\label{RSEst}
\end{aligned}
\end{align}
\end{lem}

\begin{proof}
We begin by proving estimate \eqref{RDEst}.
Observe that Theorem \ref{SobEmbeddingThm} with parameters
$m=2$, $\Mt=0$, $\Mx=2$ and $\alpha = 1$ (or
$m=2$, $\Mt=1$, $\Mx=0$ and $\alpha = 1$) yields
\begin{align*}
\norm\eta_{\WSR{1,2}{\infty}\np{\torus\times\Omega}} \leq \Cc{c} \norm\eta_{\WSR{2,4}{q}\np{\torus\times\torusnn}} \leq \Cclast{c}\norm\eta_{\ESR{q}\np\torust}
\end{align*}
for $q>2$.
Further observer that for a.e. $t\in\torus$
\begin{align*}
\norm{\Div\RD\np{t, \cdot}}_{\WSRD{-1}{q}\np\Omega} &= \sup_{\psi\in\WSRD{1}{q'}\np\Omega, \, \norm{\grad\psi}_{\LR{q'}\np\Omega}=1} \snorm{\langle \Div\RD\np{t, \cdot}, \psi\np{t, \cdot} \rangle} \\
&= \sup_{\psi\in\WSRD{1}{q'}\np\Omega, \, \norm{\grad\psi}_{\LR{q'}\np\Omega}=1} \snorm{\langle \RD\np{t, \cdot}, \grad\psi\np{t, \cdot} \rangle} \leq \norm{\RD}_{\LR{q}\np\Omega}
\end{align*}
holds.
Hence, $\tRD = \Div\RD$ fulfills
\begin{align}\label{EstRD}
\begin{aligned}
&\norm{\tRD}_{\LR{q}\np{\torus; \WSR{1}{q}\np\Omega} \cap \WSR{1}{q}\np{\torus; \WSRD{-1}{q}\np\Omega}} \leq \norm{\RD}_{\WSR{1,2}{q}\np{\torus\times\Omega}} \\
&\qquad\qquad \leq \Cc{c} \norm{\uvel}_{\WSR{1,2}{q}\np{\torus\times\Omega}} \norm{\eta}_{\WSR{1,2}{\infty}\np{\torus\times\torusnn}} \leq \Cc{c}\norm\uvel_{\WSR{1,2}{q}\np{\torus\times\Omega}}\norm\eta_{\ESR{q}\np{\torus\times\torusnn}}
\end{aligned}
\end{align}
and thus \eqref{RDEst}.
The $\LR{q}$-estimates \eqref{RFEst} and \eqref{RSEst} follow similarly by an application of H\"older's inequality and
Theorem \ref{SobEmbeddingThm}.
\end{proof}

\begin{proof}[Proof of Theorem \ref{mainThm}]
We employ the contraction mapping principle  based
on the $\LR{q}$ estimates deduced for the linear
system \eqref{LinearizedSystem}.
To this end, let
\begin{align*}
&\OA\colon\YFS\np{\torus\times\Omega}\to\SLS{q}\np{\torus\times\Omega} 
\end{align*}
be the solution operator corresponding to \eqref{LinearizedSystem} with
\begin{align*}
&\YFS\np{\torus\times\Omega} := \\
& \qquad \LR{q}\np{\torus\times\Omega}^3\times\LR{q}\bp{\torus; \WSRNN{1}{q}\np\Omega}\cap\WSR{1}{q}\bp{\torus; \WSRD{-1}{q}\np\Omega}\times\LR{q}\bp{\torus;\WSR{1-\frac{1}{q}}{q}\np\torusnn}
\end{align*}
and
\begin{align*}
\WSRNN{1}{q}\np\Omega := \setcl{\g\in\WSR{1}{q}\np\Omega}{\int_\Omega g \,\dx = 0}.
\end{align*}
By Theorem \ref{MainThmFSLin}, $\OA$ is a well-defined bounded operator.
We seek a solution to \eqref{NonLinFSOnReference} as a fixed point of the mapping
\begin{align*}
\FPM\colon\SLS{q}\np{\torus\times\Omega}\to \SLS{q}\np{\torus\times\Omega}, \qquad \FPM\np{\uvel, \upres, \eta}\coloneqq \OA\bp{\ff + \RFtilde, \tRD, \fp + \tRS}.
\end{align*}
Let $r>0$ and consider some
$\np{\uvel, \upres, \eta}\in\SLS{q}\np{\torus\times\Omega}\cap\B_r$, where $B_r$ denotes the closed ball in $\SLS{q}\np{\torus\times\Omega}$ with radius $r$.
By Lemma \ref{FSNonLinEstNLTerme} we obtain
\begin{align*}
\norm{\RFtilde}_{\LR{q}\np{\torus\times\Omega}} + \norm{\tRD}_{\LR{q}\np{\torus; \WSR{1}{q}\np\Omega}\cap\WSR{1}{q}\np{\torus; \WSRD{-1}{q}\np\Omega}} &+ \norm{\tRS}_{\LR{q}\np{\torus; \WSR{1-\frac{1}{q}}{q}\np\torusnn}} \\
&\qquad\qquad\qquad\quad\leq \Cc{c}\np{1-\varepsilon}\np{r^2 + r^3}
\end{align*}
and therefore
\begin{align*}
\norm{\FPM}_{\SLS{q}} &\leq \norm{\OA}\bp{\norm{\ff}_{\LR{q}\np{\torus\times\Omega}} + \norm{\RFtilde}_{\LR{q}\np{\torus\times\Omega}} + \norm{\tRS}_{\LR{q}\np{\torust}} \\
&\qquad + \norm{\tRD}_{\LR{q}\np{\torus; \WSR{1}{q}\np\Omega}\cap\WSR{1}{q}\np{\torus; \WSRD{-1}{q}\np\Omega}} + \norm{\fp}_{\LR{q}\np{\torust}}} \\
&\leq \Cc{c}\np{\varepsilon + r^2 + r^3}.
\end{align*}
Choosing $r=\sqrt\varepsilon$ and $\varepsilon$ sufficiently small, we have
$\Cclast{c}\np{\varepsilon + r^2 + r^3}\leq r$, in which case $\FPM$ becomes a self-mapping
on $B_r$.
To complete the proof, it remains to show that $\FPM$ is a contraction.
For this purpose we utilize Theorem \ref{SobEmbeddingThm} with $m=2$, $\Mt=1$, $\Mx=0$, $\alpha=1$ and subsequently $m=2$, $\Mt=0$, $\Mx=0$, $\alpha=2$ to deduce
\begin{align*}
\norm{\eta}_{\LR{\infty}\np{\torust}} + \norm{\pt\eta}_{\LR{\infty}\np{\torust}} \leq \Cc{c}\norm{\eta}_{\WSR{2,4}{q}\np{\torust}} \leq \Cclast{c}\norm\eta_{\ESR{q}\np{\torust}}.
\end{align*}
We thus conclude from \eqref{SmallnessEtaNLFS} that
\begin{align*}
&\normL{\grad\uvel\partial_t\eta \frac{1}{1 + \eta} - \grad\vvel\partial_t\zeta \frac{1}{1 + \zeta}}_{\LR{q}\np{\torus\times\Omega}} \leq \Cc{c}\varepsilon
\norm{\np{\uvel, \upres, \eta} - \np{\vvel, \vpres, \zeta}}_{\SLS q},
\end{align*}
and similarly
\begin{align*}
&\normL{\frac{\partial_{x_3}\upres}{1+\eta}\begin{pmatrix} \rho\gradp\eta \\ -\eta \end{pmatrix} - \frac{\partial_{x_3}\vpres}{1+\zeta}\begin{pmatrix} \rho\gradp\zeta \\ -\zeta \end{pmatrix}}_{\LR{q}\np{\torus\times\Omega}} \leq \Cc{c} \varepsilon \norm{\np{\uvel, \upres, \eta} - \np{\vvel, \vpres, \zeta}}_{\SLS q}.
\end{align*}
The additional terms in $\RFtilde$ can be estimated analogously to deduce
\begin{align*}
\norm{\RFtilde\np{\uvel, \upres, \eta} - \RFtilde\np{\vvel, \vpres, \zeta}}_{\LR{q}\np{\torus\times\Omega}} \leq \Cc{c} \varepsilon
\norm{\np{\uvel, \upres, \eta} - \np{\vvel, \vpres, \zeta}}_{\SLS q}.
\end{align*}
A straightforward calculation yields
\begin{align*}
&\norm{\tRD\np{\uvel, \upres, \eta} - \tRD\np{\vvel, \vpres, \zeta}}_{\LR{q}\np{\torus; \WSR{1}{q}\np\Omega}\cap\WSR{1}{q}\np{\torus; \WSRD{-1}{q}\np\Omega}} \leq \Cc{c} \varepsilon \norm{\np{\uvel, \upres, \eta} - \np{\vvel, \vpres, \zeta}}_{\SLS q}.
\end{align*}
In the case of $\tRS$, we further have to employ the properties
of the trace operator $\TDN$
defined in \eqref{TraceGaldi}
to deduce that
\begin{align*}
\norm{\tRS\np{\uvel, \upres, \eta} - \tRS\np{\vvel, \vpres, \zeta}}_{\LR{q}\np{\torus; \WSR{1-\frac{1}{q}}{q}\np\torusnn}} \leq \Cc{c} \varepsilon
\norm{\np{\uvel, \upres, \eta} - \np{\vvel, \vpres, \zeta}}_{\SLS q}.
\end{align*}
Collecting the estimates deduced for $\tRF$, $\tRD$ and $\tRS$,
we obtain that
\begin{align*}
\norm{\FPM\np{\uvel, \upres, \eta} - \FPM\np{\vvel, \vpres, \zeta}}_{\SLS q} \leq \Cc{c} \varepsilon \norm{\np{\uvel, \upres, \eta} - \np{\vvel, \vpres, \zeta}}_{\SLS q}.
\end{align*}
Choosing $\varepsilon$ sufficiently small, we deduce that $\FPM$ is a contracting
self-mapping. By the contraction mapping principle, existence of a fixed point for
$\FPM$ follows and completes the proof.
\end{proof}

\bibliographystyle{abbrv}

\end{document}